\documentclass[12pt,reqno]{amsart}

\usepackage[margin=1in]{geometry}
\usepackage{amsmath,amsfonts,amssymb}
\usepackage{graphics,graphicx,color}
\usepackage{cite}
\usepackage{hyperref}
\usepackage{txfonts}

\theoremstyle{plain}
\newtheorem{theorem}{Theorem}[section]
\newtheorem*{theorem*}{Theorem}
\newtheorem{proposition}{Proposition}[section]
\newtheorem{lemma}{Lemma}[section]

\theoremstyle{definition}

\theoremstyle{remark}
\newtheorem{remark}{Remark}[section]
\newtheorem*{arguments}{Heuristic arguments}

\numberwithin{equation}{section}

\newcommand{\e}{^\varepsilon}
\newcommand{\eps}{{\varepsilon}}
\newcommand{\ds}{\displaystyle}
\newcommand{\I}{\mathcal{I}\e}
\newcommand{\dist}{\mathrm{dist}}
\renewcommand{\a}{\alpha}
\renewcommand{\b}{\beta}
\renewcommand{\phi}{\varphi}
\newcommand{\cupl}{\bigcup\limits}
\newcommand{\suml}{\sum\limits}
\newcommand{\intl}{\int\limits}
\newcommand{\liml}{\lim\limits}
\newcommand{\maxl}{\max\limits}
\newcommand{\minl}{\min\limits}

\begin{document}

\noindent {\huge {Periodic elliptic operators with asymptotically
preassigned spectrum}} \bigskip

\noindent{\large {Andrii Khrabustovskyi}}\medskip

\textit{\begin{flushleft} \footnotesize Mathematical Division, B.
Verkin Institute for Low Temperature Physics and Engineering of
the National Academy of Sciences of Ukraine, Lenin avenue 47,
Kharkiv 61103, Ukraine, tel.: +38 057 3410986
\\
e-mail: andry9@ukr.net
\end{flushleft}}
\
\\
{\small\textbf{Abstract.} We deal with operators in $\mathbb{R}^n$
of the form
$$\mathbf{A}=-{1\over
\mathbf{b}(x)}\sum\limits_{k=1}^n\ds{\partial\over\partial
x_k}\left(\mathbf{a}(x){\partial \over\partial x_k}\right)$$ where
$\mathbf{a}(x),\mathbf{b}(x)$ are positive, bounded and periodic
functions. We denote by $\mathbf{L}_{\mathrm{per}}$ the set of
such operators. The main result of this work is as follows: for an
arbitrary $L>0$ and for arbitrary pairwise disjoint intervals
$(\alpha_j,\beta_j)\subset[0,L]$, $j=1,\dots,m$ ($m\in\mathbb{N}$)
we construct the family of operators $\{\mathbf{A}^\varepsilon\in
\mathbf{L}_{\mathrm{per}}\}_{\varepsilon}$ such that the spectrum
of $\mathbf{A}^\varepsilon$ has exactly $m$ gaps in $[0,L]$ when
$\varepsilon$ is small enough, and these gaps tend to the
intervals $(\alpha_j,\beta_j)$ as $\varepsilon\to 0$. The idea how
to construct the family $\left\{\mathbf{A}\e\right\}_\eps$ is
based on methods of the homogenization theory.
\medskip

}
\\ \\


\section*{\label{sec0}Introduction}

Our research is inspired by the following well-known result of Y.
Colin de Verdi\`{e}re \cite{CDV1}: for arbitrary numbers
$0=\lambda_1<\lambda_2<\dots<\lambda_m$ ($m\in\mathbb{N}$) and
$n\in\mathbb{N}\setminus\left\{1\right\}$ there is a
$n$-dimensional compact Riemannian manifold $M$ such that the
first $m$ eigenvalues of the corresponding Laplace-Beltrami
operator $-\Delta_M$ are exactly $\lambda_1,\dots,\lambda_m$. In
the work \cite{Khrab6} we obtained an analogue of this fact for
non-compact periodic manifolds: for an arbitrary $m$ pairwise
disjoint finite intervals on the positive semi-axis ($m\in
\mathbb{N}$) a periodic Riemannian manifold is constructed such
that the spectrum of the corresponding Laplace-Beltrami operator
has at least $m$ gaps, moreover the first $m$ gaps are close (in
some natural sense) to these preassigned intervals.

The goal of the present work is to solve a similar problem for the
following operators in $\mathbb{R}^n$ ($n\geq 2$):
\begin{gather*}
\mathbf{A}=-\mathbf{b}^{-1}\mathrm{div}\left(\mathbf{a}\nabla\right)=-{1\over
\mathbf{b}(x)}\suml_{k=1}^n\ds{\partial\over\partial
x_k}\left(\mathbf{a}(x){\partial \over\partial x_k}\right),\quad
\mathbf{a},\mathbf{b}\in \mathbf{H}_{\mathrm{per}}
\end{gather*}
where $\mathbf{H}_{\mathrm{per}}$ is a set of measurable real
functions in $\mathbb{R}^n$ satisfying the conditions
\begin{gather*}\mathbf{f}\in \mathbf{H}_{\mathrm{per}}:\quad\begin{cases}
\exists C^-,C^+>0:\quad C^-\leq \mathbf{f}(x)\leq C^+,\ \forall
x\in \mathbb{R}^n&\text{(boundedness from above and form below)}
\\\forall i\in \mathbb{Z}^n,\ \forall x\in \mathbb{R}^n:\quad
\mathbf{f}(x+i)=\mathbf{f}(x)& \text{(periodicity)}
\end{cases}
\end{gather*}
The operator $\mathbf{A}$ acts in the space
$L_{2,\mathbf{b}}(\mathbb{R}^n)=\left\{u\in L_2(\mathbb{R}^n),\
\|u\|^2_{L_{2,\mathbf{b}}(\mathbb{R}^n)}=\intl_{\mathbb{R}^n}|u(x)|^2
\mathbf{b}(x)dx\right\}$, it is self-adjoint and positive. We
denote by $\mathbf{L}_{\mathrm{per}}$ the set of such operators.

Operators of this type occur in various areas of physics, for
example in the case $n=3$ the operator $\mathbf{A}$ governs the
propagation of acoustic waves in a medium with periodically
varying mass density $(\mathbf{a}(x))^{-1}$ and compressibility
$\mathbf{b}(x)$.

It is well-known (see e.g. \cite{Kuchment}) that the spectrum
$\sigma(\mathbf{A})$ of the operator $\mathbf{A}\in
\mathbf{L}_{\mathrm{per}}$ has band structure, i.e.
$\sigma(\mathbf{A})$ is the union of compact intervals
$[a_k^-,a_k^+]\subset [0,\infty)$ called \textit{bands}
($a_0^-=0$, $a_k^-\underset{k\to\infty}\nearrow \infty$). In
general the bands may overlap. The open interval $(\a,\b)$ is
called a \textit{gap} if $(\a,\b)\cap
\sigma(\mathbf{A})=\varnothing$ and $\a,\b\in \sigma(\mathbf{A})$.

The main result of this work is the following


\begin{theorem}[Main Theorem]\label{th0}
Let $L>0$ be an arbitrary number and let $(\a_j,\b_j)$
($j={1,\dots,m},\ m\in \mathbb{N} $) be arbitrary intervals
satisfying
\begin{gather}\label{intervals}
0<\a_1,\quad \a_j<\b_{j}< \a_{j+1},\ j=\overline{1,m-1},\quad
\a_m<\b_{m}<L
\end{gather}
Let $n\in\mathbb{N}\setminus\{1\}$.

Then one can construct the family of functions
$\left\{\mathbf{a}\e\in \mathbf{H}_{\mathrm{per}}\right\}_{\eps}$
and the function $\mathbf{b}\in \mathbf{H}_{\mathrm{per}}$ such
that the spectrum of the operator
$\mathbf{A}\e=\mathbf{b}^{-1}\mathrm{div} (\mathbf{a}\e\nabla)$
has the following structure in the interval $[0,L]$ when $\eps$ is
small enough:
\begin{gather}\label{spec1}
\sigma(\mathbf{A}\e)\cap[0,L]=[0,L]\setminus
\left(\cupl_{j=1}^{m}(\a_j^\eps,\b_j^\eps)\right)
\end{gather}
where the intervals $(\a_j\e,\b_j\e)$ satisfy
\begin{gather}\label{spec2}
\forall j=1,\dots,m:\quad \liml_{\eps\to 0}\a_j^\eps=\a_j,\
\liml_{\eps\to 0}\b_j^\eps=\b_j
\end{gather}
Moreover, $\mathbf{a}\e(x),\ \mathbf{b}(x)$ are step-functions
having at most $m+1$ values.
\end{theorem}

\begin{remark}It follows from (\ref{intervals})-(\ref{spec2}) that the operator $\mathbf{A}\e$
has exactly $m$ gaps in $[0,L]$ when $\eps$ is small enough. In
general, the existence of gaps in the spectra of operators from
$\mathbf{L}_{\mathrm{per}}$ is not guaranteed, for instance in the
case of constant $\mathbf{a}(x)$, $\mathbf{b}(x)$ the spectrum
$\sigma(\mathbf{A})$ coincides with $[0,\infty)$. Various
operators from $\mathbf{L}_{\mathrm{per}}$ with gaps in their
spectrum were studied in the works
\cite{Figotin1,Frielander,Hempel,Zhikov,Figotin2,Figotin3,DavHar,Green,Post_JDE}
(see also the overview \cite{HempelPost}). In these works spectral
gaps are the result of high contrast either in the coefficient
$\mathbf{a}(x)$ \cite{Figotin1,Hempel,Frielander,Zhikov} or in the
coefficient $\mathbf{b}(x)$\cite{Figotin2,Figotin3} or in both
coefficients \cite{DavHar,Green,Post_JDE} (the last three works
deal with the Laplace-Beltrami operator in $\mathbb{R}^n$ with
conformally flat periodic metric; obviously, this operator belongs
to $\mathbf{L}_{\mathrm{per}}$).

The operator $\mathbf{A}\e$ constructed in the present work also
has high contrast in the coefficients (namely, $\liml_{\eps\to
0}\left(\ds{\max_{x\in \mathbb{R}^n}\mathbf{a}\e(x)\over\min_{x\in
\mathbb{R}^n} \mathbf{a}\e(x)}\right)=\infty$), but their form
essentially differs from the form of the coefficients in the works
mentioned above.

\end{remark}

The idea how to construct the functions $\mathbf{a}\e(x)$,
$\mathbf{b}(x)$ has come from the homogenization theory. We
briefly describe this construction.

\begin{figure}[h]
  \begin{center}
    \leavevmode
    \begin{picture}(0,0)
       \includegraphics{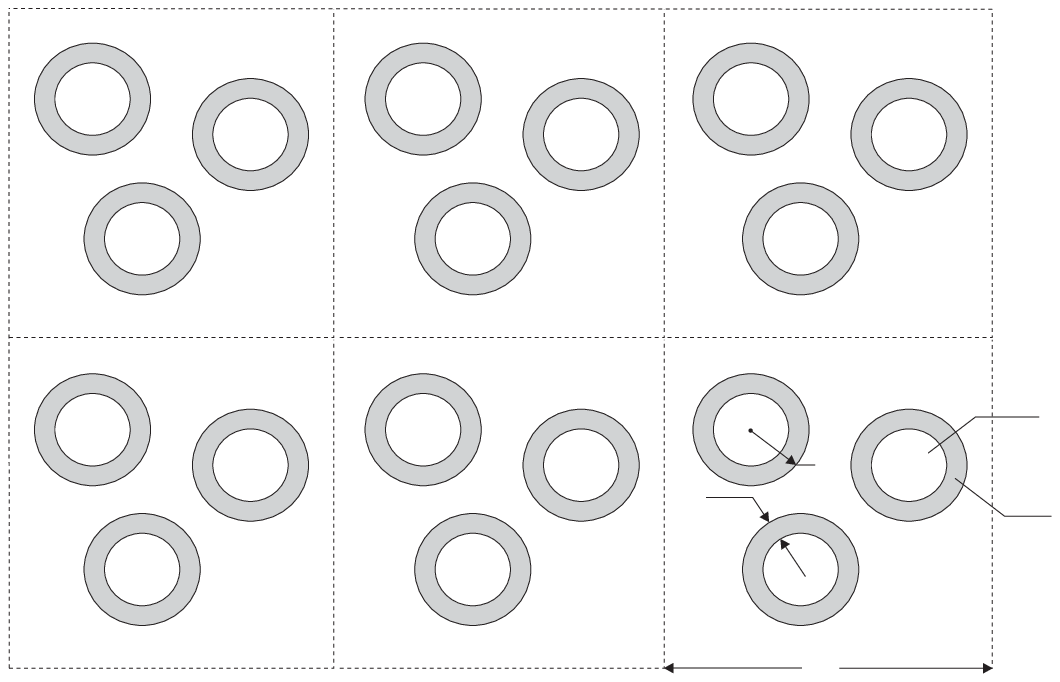}
    \end{picture}
    \setlength{\unitlength}{4144sp}%
    \begin{picture}(5000,3000)(0,0)
      \put(4670, 1100){$B_{ij}\e$}
      \put(4750, 650){$G_{ij}\e$}
      \put(3650, 900){$r\e$}
      \put(2990, 730){$d\e$}
      \put(3610, 0){$\eps$}

    \end{picture}

  \end{center}
  \caption{}\label{fig1}
\end{figure}

Let $\eps>0$ be a small number. Let $G\e=\cupl_{i\in
\mathbb{Z}^n}\cupl_{j=1}^m G_{ij}\e$ be a union of pairwise
disjoint spherical shells $G_{ij}\e$ lying in $\mathbb{R}^n$. It
is supposed that the following conditions hold (see also Fig. 1):
\begin{itemize}
\item for any fixed $j\in\left\{1,\dots,m\right\}$ the shells
$G_{ij}\e$ are centered at the nodes of $\eps$-periodic lattice in
$\mathbb{R}^n$,

\item the shells $G_{0j}\e$ ($j=1,\dots,m$) belong  to the cube
$\left\{x=(x_1,\dots,x_n)\in \mathbb{R}^n:\ 0< x_k< \eps,\ \forall
k\right\}$.

\end{itemize}
The external radius of the shells is equal to $r\e=r\eps$ ($r>0$),
the thickness of their walls is equal to $d\e=\eps^\gamma$
($\gamma>3$). By $B_{ij}\e$ we denote the sphere interior to
$G_{ij}\e$. We set $B\e=\cupl_{i\in \mathbb{Z}^n}\cupl_{j=1}^m
B_{ij}\e$.

We define the functions $a\e(x),\ b\e(x)$ by the formulae
\begin{gather}\label{ab0}
a\e(x)=\begin{cases}1,&x\in\ds \mathbb{R}^n\setminus
G\e,\\a_j\eps^{\gamma+1},&x\in G_{ij}\e,\end{cases}\quad
b\e(x)=\begin{cases}1,&x\in\ds \mathbb{R}^n\setminus \left(B\e\cup
G\e\right),\\b_j,&x\in B_{ij}\e\cup G_{ij}\e,\end{cases}
\end{gather} where $a_j$, $b_j$ ($j=1,\dots,m)$ are positive
constants, which will be chosen later on. We consider the operator
\begin{gather*}
\mathcal{A}\e=-(b\e)^{-1}\mathrm{div}\left(a\e\nabla\right)=-{1\over
b\e(x)}\suml_{k=1}^n\ds{\partial\over\partial
x_k}\left(a\e(x){\partial \over\partial x_k}\right)
\end{gather*}

It will be proved (see Theorem \ref{th1} below) that the spectrum
of $\mathcal{A}\e$ converges to the spectrum of some operator
$\mathcal{A}^0$ acting in the Hilbert space
$L_2(\mathbb{R}^n)\underset{j=\overline{1,m}}\oplus
L_{2,\rho_j/\sigma_j}(\mathbb{R}^n)$, where $\rho_j,\sigma_j$
($j=1,\dots,m$) are positive constants. The spectrum of
$\mathcal{A}^0$ coincides with the set
$[0,\infty)\setminus\left(\cupl_{j=1}^m(\sigma_j,\mu_j)\right)$,
where the intervals $(\sigma_j,\mu_j)$ satisfy
\begin{gather*}
0<\sigma_1,\quad \sigma_{j}<\mu_{j}<\sigma_{j+1},\
j=\overline{1,m-1},\quad \sigma_m<\mu_m<\infty
\end{gather*}
and depend in a special way on $a_j$ and $b_j$.

More precisely, we will prove that for an arbitrary $L>\mu_k$ the
spectrum of the operator $\mathcal{A}\e$ has the following
structure in the interval $[0,L]$ when $\eps$ is small enough:
$$\sigma(\mathcal{A}\e)\cap[0,L]=[0,L]\setminus\left(\cupl_{j=1}^m (\sigma_j\e,\mu_j\e)\right)$$
where the intervals $(\sigma_j\e,\mu_j\e)$ satisfy
\begin{gather*}
\forall j=1,\dots,m:\quad \liml_{\eps\to
0}\sigma_j\e=\sigma_j,\quad \liml_{\eps\to 0}\mu_j\e=\mu_j
\end{gather*}

Furthermore, we will prove (see Theorem \ref{th2} below) that for
arbitrary intervals $(\a_j,\b_j)$ ($j={1,\dots,m},\ m\in
\mathbb{N}$) satisfying (\ref{intervals}) one can choose such
$a_j$, $b_j$ in (\ref{ab0}) that the following equalities hold:
\begin{gather}\label{ab}
\forall j=1,\dots,m:\quad \sigma_j=\alpha_j,\ \mu_j=\beta_j
\end{gather}

Finally we set (below $y\in \mathbb{R}^n$)
$$\mathbf{a}\e(y)=\eps^{-2} a\e(x),\ \mathbf{b}(y)=b\e(x),\text{ where }x=y\eps$$
(obviously, $\mathbf{b}(y)$ is independent of $\eps$). It is clear
that $\mathbf{a}\e,\mathbf{b}$ belong to
$\mathbf{H}_{\mathrm{per}}$ and are step-functions having at most
$m+1$ values. It is easy to see that the spectra of the operator
$$\mathbf{A}\e={\mathbf{b}}^{-1}\mathrm{div}({\mathbf{a}}\e\nabla)$$
and the operator $\mathcal{A}\e$ coincide (in fact, $\mathbf{A}\e$
is obtained from $\mathcal{A}\e$ via change of variables
$x={y\eps}$).

It follows from Theorem \ref{th1}-\ref{th2} that
$\sigma(\mathbf{A}\e)$ satisfies (\ref{spec1})-(\ref{spec2}).

We remark that the gaps open up in the spectrum of $\mathbf{A}\e$
because of the high contrast in the coefficient $\mathbf{a}\e(x)$.
The coefficient $\mathbf{b}(x)$ is independent of $\eps$ and it is
needed only in order to control the behavior of the gaps as
$\eps\to 0$. In fact, the operator
$-\mathrm{div}(\mathbf{a}\e\nabla)$ also has at least $m$ gaps
when $\eps$ is small enough, but in general they do not converge
to $(\a_j,\beta_j)$ as $\eps\to 0$.

\begin{arguments}
The classical problem of the homogenization theory (see e.g.
\cite{CPS,CioDon,CioSJP,March,Sanch,Tartar,ZKO}) is to describe
the asymptotic behaviour as $\eps\to 0$ of the operator
$\mathcal{A}\e$ which acts in $L_2(\Omega)$ ($\Omega\subset
\mathbb{R}^n$ is a bounded domain) and is defined by the operation
$$\mathcal{A}_\Omega\e=-\mathrm{div}\left(a\e\nabla\right)$$
and either Dirichlet or Neumann boundary conditions on
$\partial\Omega$. Here
\begin{gather}\label{classic}
a\e(x)=\mathbf{a}(x\eps^{-1})\text{, where }\mathbf{a}\in
\mathbf{H}_{\mathrm{per}}
\end{gather}
It is well-known that $\mathcal{A}\e$ strongly resolvent converges
to the operator (so-called "homogenized operator")
$$\mathcal{A}_\Omega^0=-\suml_{k,l=1}^n \widehat{a}^{kl}\ds{\partial^2\over\partial
x_k\partial x_l}$$ where the constants $\widehat{a}^{kl}$ satisfy:
$\exists C^-,C^+>0\text{ s.t. }\forall\xi\in \mathbb{R}^n\
C^-|\xi|^2\leq \widehat{a}^{kl}\xi_k\xi_l\leq C^+|\xi|^2$.

It is interesting to study the asymptotic behaviour of the
operator $\mathcal{A}\e$ when $a\e$ has more complicated form
comparing with (\ref{classic}). In particular interest is the case
when $a\e$ is bounded below but not uniformly in $\eps$. This is
just our situation (see (\ref{ab0})): for fixed $\eps$ one has
$\minl_{x\in \mathbb{R}^n} a\e(x) >0$, but $\liml_{\eps\to
0}\left(\minl_{x\in \mathbb{R}^n} a\e(x) \right)=0$. Such type
problems were widely studied in \cite[Chapter 7]{March}. In
particular, the authors considered the operator
$\mathcal{A}_\Omega^{D,\eps}$ which acts in $L_2(\Omega)$ and is
defined by the operation
$\mathcal{A}_\Omega^{D,\eps}=-\mathrm{div}(a\e\nabla)$ and the
Dirichlet boundary conditions on $\partial\Omega$. Here
$\Omega\subset \mathbb{R}^n$ is a bounded domain, $a\e$ is defined
by (\ref{ab0}) (only the case $m=1$ was considered). It was proved
that $\mathcal{A}_\Omega^{D,\eps}$ converges as $\eps\to 0$ (in
some sense which is close to strong resolvent convergence) to the
operator $\mathcal{A}_\Omega^{D,0}$ acting in the space
$L_2(\Omega)\oplus L_{2,\rho/\sigma}(\Omega)$ and being defined by
the operation
\begin{gather}\label{operation}
 \mathcal{A}_\Omega^{D,0}= \left(\begin{matrix} \ds
 - \widehat{a}\Delta +\rho\ &\ -\rho\\
-\sigma\ &\ \sigma
\end{matrix}\right)
\end{gather}
and the definitional domain
$\mathcal{D}(\mathcal{A}_\Omega^{D,0})=\left\{(u,v)\in
H^2(\Omega)\oplus L_{2,\rho/\sigma}(\Omega):\
u|_{\partial\Omega}=0\right\}$. Here $\widehat{a},\rho,\sigma$ are
positive constants that do not depend on $\Omega$. A similar
result is valid for the operator $\mathcal{A}_\Omega^{N,\eps}$
(the superscripts $"D"$ and $"N"$ mean Dirichlet and Neumann
boundary conditions): the corresponding homogenized operator
$\mathcal{A}_\Omega^{N,0}$ is defined by operation
(\ref{operation}) and the definitional domain
$\mathcal{D}(\mathcal{A}_\Omega^{N,0})=\left\{(u,v)\in
H^2(\Omega)\oplus L_{2,\rho/\sigma}(\Omega):\ \left.{\partial
u\over\partial n}\right|_{\partial\Omega}=0\right\}$.

Although in general the strong resolvent convergence of operators
does not imply the Hausdorff convergence of their spectra (see the
definition at the beginning of Section \ref{sec5}), but suppose
for a moment that this is true for the operators
$\mathcal{A}_{\Omega}^{D,\eps}$ and
$\mathcal{A}_{\Omega}^{N,\eps}$, i.e.\footnote{We will prove this
statement in Section \ref{sec5} (the only difference is that we
will consider quasi-periodic boundary conditions, but for
Dirichlet and Neumann boundary conditions the proof is similar.)}
\begin{gather*}
\sigma(\mathcal{A}_{\Omega}^{D,\eps})\underset{\eps\to 0}\to
\sigma(\mathcal{A}_{\Omega}^{D,0}),\
\sigma(\mathcal{A}_{\Omega}^{N,\eps})\underset{\eps\to 0}\to
\sigma(\mathcal{A}_{\Omega}^{N,0})\text{ in the Hausdorff sense}
\end{gather*}
We denote $\Omega_R=\left\{x\in \mathbb{R}^n:\ |x|<R\right\}$. One
can prove (for example, it follows from \cite[Proposition
2.3]{Khrab3}) that
\begin{gather*}
\forall\Omega\subset \mathbb{R}^n:\
(\sigma,\mu)\cap\sigma(\mathcal{A}_\Omega^{D/N,0})=\varnothing\\
\forall [d^-,d^+]\subset[0,\infty)\setminus(\sigma,\mu)\quad
\exists R_d>0:\
\sigma(\mathcal{A}_{\Omega_R}^{D/N,0})\cap[d^-,d^+]\not=
\varnothing\text{ for }R>R_{d}
\end{gather*}
where $D/N$ is either $D$ or $N$, $\mu=\sigma+\rho$. These suggest
that when $\eps$ is small enough the operator $\mathcal{A}^{\eps}$
has a gap in the spectrum and this gap tends to the interval
$(\sigma,\mu)$ as $\eps\to 0$.

The close problem was also considered in \cite{Pankratov} where
the authors studied the asymptotic behaviour of the attractors for
semilinear hyperbolic equation ${\partial^2_{tt}
u}+\mathcal{A}_{\Omega}^{D,\eps}u+f\e(u)=h\e$.

We remark that the proof of the resolvent convergence in
\cite{March} is based on the method of so-called "local energy
characteristics". This method is well adapted for both periodic
and non-periodic operators but it is quite cumbersome. Therefore
in the present work following \cite{Khrab6} we carry out the proof
in more simple fashion via the substitution of a suitable test
function into the variational formulation of the spectral problem.

\end{arguments}

In the next section we describe precisely the operator
$\mathcal{A}\e$ and formulate Theorems \ref{th1}-\ref{th2}. Their
proofs are carried out in Sections \ref{sec2}-\ref{sec7}.


\section{\label{sec1}Construction of operators $\mathcal{A}\e$ and main results}

Let $n\in \mathbb{N}\setminus\{1\}$, $m\in \mathbb{N}$. Let the
points $x_j\in \mathbb{R}^n$ ($j=1,\dots, m$) and the number $r>0$
be such that the closed balls ${B}_j=\left\{x\in \mathbb{R}^n:\
|x-x_j|\leq r\right\}$ are pairwise disjoint and belong to the
open cube $${{{Y}}}=\left\{x=(x_1,\dots,x_n)\in \mathbb{R}^n:\ 0<
x_k< 1,\ \forall k\right\}$$

Let $\eps>0$. We introduce the following notations (below $i\in
\mathbb{Z}^n$, $j=1,\dots,m$):
\begin{gather*}
x_{ij}\e=\eps(x_j+i)\\
{G}_{ij}\e=\left\{x\in \mathbb{R}^n:\ r\e-d\e<|x-x_{ij}\e|<
r^\eps\right\},\quad {B}_{ij}\e=\left\{x\in \mathbb{R}^n:\
|x-x_{ij}\e|< r\e-d\e\right\}
\end{gather*}
where
\begin{gather*}
 r\e=r\eps,\quad d\e=\eps^\gamma,\ \gamma>3
\end{gather*}
We also denote
$$G\e=\cupl_{i\in\mathbb{Z}^n}\cupl_{j=1}^m G_{ij}\e,\quad
B\e=\cupl_{i\in\mathbb{Z}^n}\cupl_{j=1}^m B_{ij}\e,\quad
F\e=\mathbb{R}^n\setminus\left(\overline{G\e\cup B\e}\right)$$

We define the piecewise constant functions $a\e(x),\ b\e(x)$ by
the formulae
\begin{gather}\label{a}
a\e(x)=\begin{cases}1,&x\in F\e\cup B\e,\\a_j\e\equiv
a_j\eps^{\gamma+1},&x\in G_{ij}\e\quad (i\in \mathbb{Z}^n,\
j=1,\dots,m),\\\end{cases}\\\label{b} b\e(x)=\begin{cases}1,&x\in
F\e,\\b_j,&x\in B_{ij}\e\cup G_{ij}\e\quad (i\in \mathbb{Z}^n,\
j=1,\dots,m),\end{cases}
\end{gather} where $a_j$, $b_j$ ($j=1,\dots,m)$ are positive
constants.

Now we define precisely the operator $\mathcal{A}\e$. By
$L_{2,{b\e}}(\mathbb{R}^n)$ we denote the Hilbert space of
functions from $L_2(\mathbb{R}^n)$ with the following scalar
product:
\begin{gather*}
(u,v)_{L_{2,{b\e}}(\mathbb{R}^n)}=\intl_{\mathbb{R}^n}u(x)\overline{v(x)}b\e(x)dx,\quad
\end{gather*}
Remark that
\begin{gather}\label{equiv}
C^{-}\|\cdot\|_{L_{2}(\mathbb{R}^n)}\leq
\|\cdot\|_{L_{2,{b\e}}(\mathbb{R}^n)}\leq C^+
\|\cdot\|_{L_{2}(\mathbb{R}^n)}
\end{gather}
where the positive constants $C^{-},C^+$ are independent of
$\eps$. By $\eta_{\mathbb{R}^n}\e[u,v]$ we denote the sesquilinear
form in ${L_{2,{b\e}}(\mathbb{R}^n)}$ which is defined by the
formula
\begin{gather*}
\eta_{\mathbb{R}^n}\e[u,v]=\intl_{\mathbb{R}^n}a\e(x)\left(\nabla
u,\nabla \overline{v}\right)dx
\end{gather*}
with $\mathrm{dom}(\eta_{\mathbb{R}^n}\e)=H^1(\mathbb{R}^n)$. Here
$\left(\nabla u,\nabla
\overline{v}\right)=\ds\suml_{k=1}^n{\partial u\over\partial
x_k}{\partial \overline{v}\over\partial x_k}$. The form is densely
defined, closed and positive. Then (see e.g. \cite{Kato}) there
exists the unique self-adjoint and positive operator
$\mathcal{A}\e$ associated with the form
$\eta\e_{\mathbb{R}^n}[u,v]$, i.e.
\begin{gather}
\label{op_fo} (\mathcal A\e u,v)_{L_{2,{b\e}}(\mathbb{R}^n)}=
\eta_{\mathbb{R}^n}\e[u,v],\quad\forall u\in
\mathrm{dom}(\mathcal{A}\e),\ \forall v\in
\mathrm{dom}(\eta\e_{\mathbb{R}^n})
\end{gather}
Its domain $\mathrm{dom}(\mathcal{A}\e)$ consists of functions $u$
belonging to the spaces $H^2(F\e)$, $H^2(G_{ij}\e)$,
$H^2(B_{ij}\e)$ (for any $i\in\mathbb{Z}^n$, $j=1,\dots,m$) and
satisfying the following conditions on the boundaries of the
shells $G_{ij}\e$:
\begin{gather}\label{sopr}
\begin{cases} (u)^+=(u)^-\text{\quad and\quad } \ds\left({\partial
u\over\partial n}\right)^+=a_j\e\left({\partial u\over\partial
n}\right)^-,
&x\in\partial \left(\overline{B_{ij}\e}\cup G_{ij}\e\right),\\
(u)^+=(u)^-\text{\quad and\quad } a_j\e\ds\left({\partial
u\over\partial n}\right)^+=\left({\partial u\over\partial
n}\right)^-, &x\in\partial B_{ij}\e
\end{cases}
\end{gather}
where by $+$ (\textit{resp.} $-$) we denote the values of the
function $u$ and its normal derivative on the exterior
(\textit{resp.} interior) side of either
$\partial\left(\overline{B_{ij}\e}\cup G_{ij}\e\right)$ or
$\partial B_{ij}\e$. For sufficiently smooth $u$ the operator
$\mathcal{A}\e$ is defined locally by the formula
\begin{gather}\label{local}
\mathcal{A}\e u=-{1\over
b\e(x)}\suml_{k=1}^n\ds{\partial\over\partial
x_k}\left(a\e(x){\partial u\over\partial x_k}\right)
\end{gather}

By $\sigma(\mathcal{A}\e)$ we denote the spectrum of the operator
$\mathcal{A}\e$. In order to describe the behaviour of
$\sigma(\mathcal{A}\e)$ as $\eps\to 0$ we introduce some
additional notations.

In the domain $F={{{Y}}}\setminus \cupl_{j=1}^m {B}_j$ we consider
the following problem (below $k=1,\dots,n$):
\begin{gather}\label{cell}
\left\{\begin{array}{l} \Delta v_k=0,\ x\in F\\\ds{\partial
v_k\over\partial n}=n_k,\ x\in
\partial\left(\cupl_{j=1}^m{B}_j\right)\\v_k,D v_k\text{ are
}Y\text{-periodic, i.e. } \forall\a=\overline{1,n}:\
\begin{cases}
v_k(x)=v_k(x+e_\a)\\\ds {\partial v_k\over\partial
x_\a}(x)={\partial v_k\over\partial x_\a}(x+e_\a)
\end{cases}\text{for }x=\underset{^{\overset{\qquad\quad\uparrow}{\qquad\quad
\a\text{-th place}}\qquad }}{(x_1,x_2,\dots,0,\dots,x_n)}
\end{array}\right.
\end{gather}
where $n=(n_1,\dots,n_n)$ is the outward unit normal to
$\cupl_{j=1}^m{B}_j$,
$e_\a=\underset{^{\overset{\qquad\quad\uparrow}{\qquad\quad
\a\text{-th place}}\qquad }}{(0,0,\dots,1,\dots,0)}$. It is known
(see e.g. \cite{CioSJP}) that the unique (up to a constant)
solution $v_k(x)$ of this problem exists. We denote
$$\widehat{a}^{kl}={1\over |F|}\intl_{F}\left(\nabla (x_k-v_k),\nabla(x_l-v_l)\right)dx,\ k,l=1,\dots,n$$
The matrix $\widehat{A}=\left\{\widehat{a}^{kl}\right\}$ is
symmetric and positively defined (see e.g. \cite[Chapter 1,
Proposition 2.6]{CioSJP}).

\begin{remark} In the case when $m=1$ and the center of ball $B_1$ coincides
with the center of the cube ${Y}$ the matrix
$\widehat{A}=\left\{\widehat{a}^{kl}\right\}$ has more simple
form, namely $\widehat{A}=\widehat{a}\mathrm{I}$ where
$\mathrm{I}$ is the identity matrix, $\widehat{a}>0$. This follows
easily from the symmetry of the domain $F$.
\end{remark}

We denote
\begin{gather}\label{sigmarho}
\sigma_j={na_j\over r b_j},\quad \rho_j={a_j |\partial B_j|\over
|F|}
\end{gather}
We assume that the numbers $a_j$ and $b_j$ in (\ref{a})-(\ref{b})
are such that $\sigma_i\not=\sigma_j$ if $i\not=j$. For
definiteness we suppose that $\sigma_j<\sigma_{j+1}$,
$j=1,\dots,n-1$.

And finally let us consider the following equation (with unknown
$\lambda\in\mathbb{C}$):
\begin{gather}\label{mu_eq}
\mathcal{F}(\lambda)\equiv
1+\suml_{j=1}^m{\rho_j\over\sigma_j-\lambda}=0
\end{gather}
It is easy to prove (see Section \ref{sec4}) that this equation
has exactly $m$ roots $\mu_j$ ($j=1,\dots,m$), they are real,
moreover they interlace with $\sigma_j$, i.e.
$$
\sigma_j<\mu_j<\sigma_{j+1},\ j=\overline{1,m-1},\quad
\sigma_m<\mu_m<\infty
$$

Now we are able to formulate the theorem describing the behaviour
of $\sigma(\mathcal{A}\e)$ as $\eps\to 0$.


\begin{theorem}\label{th1}
Let $L$ be an arbitrary number such that $L>\mu_m$. Then the
spectrum $\sigma(\mathcal{A}\e)$ of the operator $\mathcal{A}\e$
has the following structure in $[0,L]$ when $\eps$ is small
enough:
\begin{gather}\label{th1_f1}
\sigma(\mathcal{A}\e)\cap[0,L]=[0,L]\setminus
\left(\cupl_{j=1}^{m}(\sigma_j\e,\mu_j\e)\right)
\end{gather}
where the intervals $(\sigma_j\e,\mu_j\e)$ satisfy
\begin{gather}\label{th1_f2}
\forall j=1,\dots,m:\quad \lim_{\eps\to
0}\sigma_j\e=\sigma_j,\quad \lim_{\eps\to 0}\mu_j\e=\mu_j
\end{gather}

The set
$[0,\infty)\setminus\left(\cupl_{j=1}^m(\sigma_j,\mu_j)\right)$
coincides with the spectrum $\sigma(\mathcal{A}^0)$ of the
self-adjoint operator $\mathcal{A}^0$ which acts in the space
$L_2(\mathbb{R}^n)\underset{j=\overline{1,m}}\oplus
L_{2,{\rho_j/\sigma_j}}(\mathbb{R}^n)$ and is defined by the
formula
\begin{gather*}
\mathcal{A}^0 U= \left(\begin{matrix} -\ds\suml_{k,l=1}^n
\widehat{a}^{kl} \ds{\partial^2 u\over\partial x_k\partial x_l}
+\ds
\suml_{j=1}^m{\rho_j}(u-u_j)\\
\sigma_1(u_1-u)\\\sigma_2(u_2-u)\\\dots\\\sigma_m(u_m-u)
\end{matrix}\right),\ U=\left(\begin{matrix}u\\u_1\\u_2\\\dots\\u_m\end{matrix}\right)\in
\mathrm{dom}(\mathcal{A}^0)={H^2(\mathbb{R}^n)}\underset{j=\overline{1,m}}\oplus
L_{2,\rho_j/\sigma_j}(\mathbb{R}^n)
\end{gather*}
\end{theorem}

To complete the proof of Theorem \ref{th0} we have to choose such
$a_j$ and $b_j$ in (\ref{a}), (\ref{b}) that (\ref{ab}) holds.


\begin{theorem}\label{th2}
Let $(\a_j,\b_j)$ ($j={1,\dots,m}$) be arbitrary intervals
satisfying (\ref{intervals}).

Then (\ref{ab}) holds if we choose
\begin{gather}\label{exact_formula}
a_j=\ds {|F|\over|\partial
{B}_j|}{({\b_j-\a_j})\prod\limits_{i=\overline{1,m}|i\not=
j}\ds\left({\b_i-\a_j\over \a_i-\a_j}\right)},\quad b_j=\ds
{n|F|\over r|\partial
{B}_j|}{{\b_j-\a_j\over\a_j}\prod\limits_{i=\overline{1,m}|i\not=
j}\ds\left({\b_i-\a_j\over \a_i-\a_j}\right)}
\end{gather}
\end{theorem}

\begin{remark} Since the intervals $(\a_j,\b_j)$ satisfy (\ref{intervals})
then
\begin{gather*}
\forall j:\ \b_j>\a_j,\quad \forall i\not= j:\
\mathrm{sign}(\b_i-\a_j)=\mathrm{sign}(\a_i-\a_j)\not= 0
\end{gather*}
Therefore $({\b_j-\a_j})\prod\limits_{i=\overline{1,m}|i\not=
j}\ds\left({\b_i-\a_j\over \a_i-\a_j}\right)>0$ and thus the
choice of $a_j$ and $b_j$ is correct.\end{remark}

The scheme of the proof of these theorems is as follows.

In Section \ref{sec2} we introduce the functional spaces and
operators that are used throughout the proof. Also we present
well-known results describing the spectrum of the operator
$\mathcal{A}\e$

In Section \ref{sec3} we prove several technical lemmas.

In Section \ref{sec4} we show that
\begin{gather}\label{Aspectrum}
\sigma(\mathcal{A}^0)=[0,\infty)\setminus\left(\cupl_{j=1}^m(\sigma_j,\mu_j)\right)
\end{gather}

Section \ref{sec5} is a crucial part of the proof: we show that as
$\eps\to 0$ the set $\sigma(\mathcal{A}\e)$ converges in the
Hausdorff sense to the set $\sigma(\mathcal{A}^0)$.

In Section \ref{sec6} we prove that for an arbitrary $L>0$ the
spectrum $\sigma(\mathcal{A}\e)$ has at most $m$ gaps within the
interval $[0,L]$ when $\eps$ is small enough. Together with the
Hausdorff convergence this fact implies the statements of Theorem
\ref{th1}.

And finally in Section \ref{sec7} we prove Theorem \ref{th2}.

\begin{remark}
We present the proof of Theorem \ref{th1} for the case $n\geq 3$
only. For the case $n=2$ the proof is repeated word-by-word with
some small modifications (for example in formula (\ref{hat_v})
below $r^{2-n}$ has to be replaced by $\ln r$).
\end{remark}


\section{\label{sec2}Preliminaries: functional spaces and operators}

Below $\Omega$ is a domain in $\mathbb{R}^n$ with Lipschitz
boundary (if $\partial\Omega\not=\varnothing$), for simplicity we
suppose that $\partial\Omega\cap
\overline{\cupl_{i,j}G_{ij}\e}=\varnothing$. Throughout the paper
we will use the following functional spaces:

\begin{itemize}
\item $L_{2,b\e}(\Omega)$ be the Hilbert space of functions from
$L_2(\Omega)$ with the scalar product
$$(u,v)_{L_{2,b\e}(\Omega)}=\ds\intl_\Omega u(x)\overline{v(x)}b\e(x)dx$$

\item $\overset{\circ}{H}{}^1(\Omega)$ be the subspace of
$H^1(\Omega)$ consisting of functions vanishing on
$\partial\Omega$,

\item $\overset{\circ}{C}{}^\infty(\Omega)$ be the space of
functions from $C^\infty({\Omega})$ compactly supported in
$\Omega$,

\item $H^{2,\eps}(\Omega)$ be the space of functions belonging to
$H^2(\Omega\cap G_{ij}\e)$, $H^2(\Omega\cap B_{ij}\e)$ ($i\in
Z^n$, $j=1,\dots,m$), $H^2(\Omega\cap F\e)$ and satisfying
conditions (\ref{sopr}) for all shells $G_{ij}\e$ belonging to
$\Omega$,

\item $C^{2,\eps}(\Omega)$ be the space of functions belonging to
$C^2(\Omega\cap G_{ij}\e)$, $C^2(\Omega\cap B_{ij}\e)$ ($i\in
Z^n$, $j=1,\dots,m$), $C^2(\Omega\cap F\e)$ and satisfying
conditions (\ref{sopr}) for all shells $G_{ij}\e$ belonging to
$\Omega$.

\end{itemize}\smallskip

For $u,v\in H^1(\Omega)$ we denote
\begin{gather}
\label{form} \eta\e_\Omega[u,v]=\intl_{\Omega}a\e(x)\left(\nabla
u,\nabla \bar{v}\right)dx
\end{gather}
By $\eta^{N,\eps}_{\Omega}$ (\textit{resp.}
$\eta^{D,\eps}_{\Omega}$) we denote the sesquilenear form defined
by formula (\ref{form}) and the definitional domain $H^1(\Omega)$
(\textit{resp.} $\overset{\circ}{H}{}^1(\Omega)$).

Similarly to the operator $\mathcal{A}\e$ (see (\ref{op_fo})) we
define the operator $\mathcal{A}^{N,\eps}_{\Omega}$
(\textit{resp.} $\mathcal{A}^{D,\eps}_{\Omega}$) as the operator
acting in $L_{2,b\e}(\Omega)$ and associated with the form
$\eta^{N,\eps}_{\Omega}$ (\textit{resp.} $\eta^{D,\eps
}_{\Omega}$). The definitional domain
$\mathrm{dom}(\mathcal{A}^{N,\eps}_{\Omega})$ (\textit{resp.}
$\mathrm{dom}(\mathcal{A}^{D,\eps}_{\Omega})$) consists of
functions from $H^{2,\eps}(\Omega)$ satisfying the condition
$\left.{\partial u\over \partial n}\right|_{\partial\Omega}=0$
(\textit{resp.} $u|_{\partial\Omega}=0$) that justifies the upper
index "N" (\textit{resp.} "D") which indicates the Neumann
(\textit{resp.} Dirichlet) boundary conditions.

The spectra of the operators $\mathcal{A}^{N,\eps}_{\Omega}$,
$\mathcal{A}^{D,\eps}_{\Omega}$ are purely discrete. We denote by
$\left\{\lambda_k^{N,\eps}(\Omega)\right\}_{k\in\mathbb{N}}$
(\textit{resp.}
$\left\{\lambda_k^{D,\eps}(\Omega)\right\}_{k\in\mathbb{N}}$) the
sequence of eigenvalues of $\mathcal{A}^{N,\eps}_{\Omega}$
(\textit{resp.} $\mathcal{A}^{D,\eps}_{\Omega}$) written in the
increasing order and repeated according to their multiplicity.
\medskip

Now let us describe the structure of the spectrum
$\sigma(\mathcal{A}\e)$ of the operator $\mathcal{A}\e$.  The
operator $\mathcal{A}\e$ is periodic with respect to the periodic
cell $${Y}_0\e=\left\{x\in \mathbb{R}^n:\ 0< x_k<\eps,\ \forall
k\right\}$$ We denote
$\mathbb{T}^n=\left\{\theta=(\theta_1,\dots,\theta_n)\in
\mathbb{C}^n:\ |\theta_k|=1,\ \forall k\right\}$. For $\theta\in
\mathbb{T}^n$ we introduce the functional space
$H_\theta^1({{{Y}}}_0\e)$ consisting of functions from
$H^1({{{Y}}}_0\e)$ that satisfy the following condition on
$\partial{{{Y}}}_0\e$:
\begin{gather}\label{theta1}
\forall k=\overline{1,n}:\quad u(x+\eps e_k)=\theta_k u(x)\text{
for }x=\underset{^{\overset{\qquad\quad\uparrow}{\qquad\quad
k\text{-th place}}\qquad }}{(x_1,x_2,\dots,0,\dots,x_n)}
\end{gather}
where $e_k={(0,0,\dots,1,\dots,0)}$.

By $\eta^{\theta,\eps}_{{{{Y}}}_0\e}$ we denote the sesquilenear
form defined by formula (\ref{form}) (with ${{{Y}}}_0\e$ instead
of $\Omega$) and the definitional domain
$H_\theta^1({{{Y}}}_0\e)$.

We define the operator $\mathcal{A}^{\theta,\eps}_{{{{Y}}}_0\e}$
as the operator acting in $L_{2,b\e}({{{{Y}}}_0\e})$ and
associated with the form $\eta^{\theta,\eps}_{{{{Y}}}_0\e}$. Its
definitional domain
$\mathrm{dom}(\mathcal{A}^{\theta,\eps}_{{{{Y}}}_0\e})$ consists
of the functions from $H^{2,\eps}({Y}_0\e)$ satisfying the
condition (\ref{theta1}) and the condition
\begin{gather*}
\forall k=\overline{1,n}:\quad {\partial u\over\partial
x_k}(x+\eps e_k)=\theta_k {\partial u\over\partial x_k}(x)\text{
for }x=\underset{^{\overset{\qquad\quad\uparrow}{\qquad\quad
k\text{-th place}}\qquad }}{(x_1,x_2,\dots,0,\dots,x_n)}
\end{gather*}

The operator $\mathcal{A}^{\theta,\eps}_{{{{Y}}}_0\e}$ has purely
discrete spectrum. We denote by
$\left\{\lambda_k^{\theta,\eps}({{{Y}}}_0\e)\right\}_{k\in\mathbb{N}}$
the sequence of eigenvalues of
$\mathcal{A}^{\theta,\eps}_{{{{Y}}}_0\e}$ written in the
increasing order and repeated according to their multiplicity.

From the min-max principle (see e.g. \cite{Reed}) and the
enclosure $H^1({{{Y}}}_0\e) \supset H^1_\theta({{{Y}}}_0\e)\supset
\overset{\circ}{H}{}^1({{{Y}}}_0\e)$ one can easily obtain the
inequality
\begin{gather}\label{enclosure}
\forall k\in \mathbb{N}:\quad
\lambda_k^{N,\eps}({{{Y}}}_0\e)\leq\lambda_k^{\theta,\eps}({{{Y}}}_0\e)\leq
\lambda_k^{D,\eps}({{{Y}}}_0\e)
\end{gather}

The following fundamental result (see e.g. \cite{Kuchment})
establishes the relationship between the spectra of the operators
$\mathcal{A}\e$ and $\mathcal{A}^{\theta,\eps}_{{{{Y}}}_0\e}$.

\begin{theorem*} One has
\begin{gather}\label{repres1}
\sigma(\mathcal{A}\e)=\cupl_{k=1}^\infty
\mathcal{J}_k(\mathcal{A}\e)
\end{gather}
where $\mathcal{J}_k(\mathcal{A}\e)=\cupl_{\theta\in \mathbb{T}^n}
\left\{\lambda_k^{\theta,\eps}({{{Y}}}_0\e)\right\}$. The sets
$\mathcal{J}_k(\mathcal{A}\e)$ are compact intervals.
\end{theorem*}

\begin{remark}\label{rem21} It is clear that if $\eps^{-1}\in
\mathbb{N}$ then $\mathcal{A}\e$ is also ${{{Y}}}$-periodic
operator, i.e. $a\e(x+i)=a\e(x)$, $b\e(x+i)=b\e(x)$ for any $i\in
\mathbb{Z}^n$, $x\in\mathbb{R}^n$. So in this case we have an
analogous representation
\begin{gather}\label{repres2}
\sigma(\mathcal{A}\e)=\cupl_{k=1}^\infty
\mathcal{\hat{J}}_k(\mathcal{A}\e)
\end{gather}
where $\mathcal{\hat{J}}_k(\mathcal{A}\e)=\cupl_{\theta\in
\mathbb{T}^n} \left\{\lambda_k^{\theta,\eps}({{{Y}}})\right\}$,
$\lambda_k^{\theta,\eps}({{{Y}}})$ is the $k$-th eigenvalue of the
operator $\mathcal{A}^{\theta,\eps}_{{{{Y}}}}$ which acts in
$L_{2,b\e}({{{Y}}})$ and is defined by the operation (\ref{local})
and the definitional domain
\begin{gather*}
\mathrm{dom}(\mathcal{A}^{\theta,\eps}_{{{{Y}}}})=\left\{u\in
H^{2,\eps}({{{Y}}}):\ \forall k=\overline{1,n}\
\begin{cases}u(x+e_k)=\theta_k u(x)\\\ds{\partial u\over\partial x_k}(x+ e_k)=\theta_k
{\partial u\over\partial x_k}(x)\end{cases}\text{for }
x=\underset{^{\overset{\qquad\quad\uparrow}{\qquad\quad k\text{-th
place}}\qquad }}{(x_1,x_2,\dots,0,\dots,x_n)}\right\}
\end{gather*}

Studying the Hausdorff convergence of $\sigma(\mathcal{A}\e)$ as
$\eps\to 0$ we will use the representation (\ref{repres2}), while
estimating the number of gaps in the interval $[0,L]$ we will use
the representation (\ref{repres1}).
\end{remark}


\section{\label{sec3}Auxiliary lemmas}

In this section we prove some technical lemmas. In order to
formulate them we introduce some additional notations.

We denote
$$\kappa={1\over 2}\minl_{j=\overline{1,m}}
\dist\left({B}_j,\partial{{{Y}}}\cup\left(\cupl_{k\not=
j}{B}_k\right)\right)$$ Recall that the closed balls ${B}_j$ are
pairwise disjoint and belong to the open cube ${{{Y}}}$, hence
$\kappa>0$.

We introduce the following sets (below $i\in \mathbb{Z}^n$,
$j=1,\dots,m$):
\begin{itemize}

\item ${{{Y}}}_i\e=\left\{x=(x_1,\dots,x_n)\in \mathbb{R}^n:\
i\eps< x_k<(i+1)\eps,\ \forall k\right\}$

\item $F\e_{i}={{{Y}}}_i\e\setminus\cupl_{j=1}^m
\left(\overline{B_{ij}\e\cup G_{ij}\e}\right)$

\item $R_{ij}\e=\left\{x\in \mathbb{R}^n:\ r\e< |x-x_{ij}\e|<
r\e+\kappa\eps\right\}$

\item $D_{ij}\e=\left\{x\in \mathbb{R}^n:\ |x-x_{ij}\e|<
r\e+\kappa\eps\right\}=B_{ij}\e\cup \overline{G_{ij}\e}\cup
R_{ij}\e$

\item $S_{ij}\e=\left\{x\in \mathbb{R}^n:\ |x-x_{ij}\e|=
r\e+\kappa\eps\right\}=\partial D_{ij}\e$

\item $\hat C_{ij}\e=\left\{x\in \mathbb{R}^n:\ |x-x_{ij}\e|=
r\e\right\}$

\item $\check C_{ij}\e=\left\{x\in \mathbb{R}^n:\ |x-x_{ij}\e|=
r\e-d\e\right\}$

\end{itemize}
We also denote
$$\I=\left\{i=(i_1,\dots,i_n)\in \mathbb{Z}^n:\ 0\leq
i_k\leq (\eps^{-1}-1), \forall k \right\}$$ and set
\begin{gather*}
G_{Y}\e=\cupl_{i\in \I}\cupl_{j=1}^m G_{ij}\e,\quad
B_{Y}\e=\cupl_{i\in \I}\cupl_{j=1}^m B_{ij}\e,\quad
F_{{{Y}}}\e=\cupl_{i\in\I} F_{i}\e
\end{gather*}
Remark that if $\eps^{-1}\in \mathbb{N}$ then
$\overline{{{{Y}}}}=\cupl_{i\in\I}\overline{{{{Y}}}_i\e}$.

By $\langle u\rangle_B$ we denote the average value of the
function $u$ over the domain $B\subset \mathbb{R}^n$ (if $|B|\not=
0$), i.e. $\langle u \rangle_{B}=\ds{1\over |B|}\intl_{B}u(x) dx$.
If $\Sigma\subset \mathbb{R}^n$ is a $(n-1)$-dimensional surface
then the Euclidean metrics in $\mathbb{R}^n$ induces on $\Sigma$
the Riemannian metrics and measure. We denote by $ds$ the density
of this measure. Again by $\langle u\rangle_\Sigma$ we denote the
average value of the function $u$ over $\Sigma$, i.e $\langle
u\rangle_\Sigma=\ds{1\over |\Sigma|}\intl_{\Sigma}u ds$ (here
$|\Sigma|=\intl_\Sigma ds$).

If $\eta[u,v]$ is a sesquilinear form then we preserve the same
notation $\eta$ for the corresponding quadratic form, i.e
$\eta[u]=\eta[u,u]$.

By $\chi_{_{\Omega}}$ we denote an indicator function of the
domain $\Omega$, i.e. $\chi_{_{\Omega}}(x)=1$ for $x\in\Omega$ and
$\chi_{_{\Omega}}(x)=0$ otherwise.

In what follows by $C,C_1...$ we denote generic constants that do
not depend on $\eps$.

\begin{lemma}\label{lm20}
 Let $D$ be a convex domain in $\mathbb{R}^n$,
$d$ be the diameter of $D$, $X$ and $Y$ be arbitrary measurable
subsets of $D$. Then for any $v\in H^1(D)$ the following
inequality holds:
\begin{gather*}
\left|\langle v\rangle _X-\langle v\rangle_Y\right|^2\leq
C\|\nabla v\|^2_{L_2(D)}{d^{n+2}\over |X|\cdot |Y|}
\end{gather*}
\end{lemma}

\begin{proof}The lemma
is proved in a similar way as Lemma 4.9 from \cite[p.117]{March}.
\end{proof}


\begin{lemma}\label{lm21} Let $\eps=\eps_N={1\over N}$, $N=1,2,3\dots$ Let $v\e\in
H^1({{{Y}}})$, $\|v\e\|^2_{H^1({{{Y}}})}<C$, $v\e\underset{\eps\to
0}\to v\in H^1({{{Y}}})$ strongly in $L_2({{{Y}}})$. Then $\forall
j=\overline{1,m}$:
\begin{gather}\label{lm21_i1}
\suml_{i\in \mathcal{I}\e} \langle
v\e\rangle_{S_{ij}\e}{\chi}_{_{{Y}_i\e}}\underset{\eps\to 0}\to
v\text{ strongly in }L_2({{{Y}}})
\\\label{lm21_i3}\suml_{i\in \mathcal{I}\e}
\langle v\e\rangle_{F_{i}\e}{\chi}_{_{{Y}_i\e}}\underset{\eps\to
0}\to v\text{ strongly in }L_2({{{Y}}})
\end{gather}

\end{lemma}


\begin{proof}For an arbitrary $i\in\I$ and $j\in\{1,\dots,m\}$ one has the
following inequalities:
\begin{gather}\label{ineq1}\left\|v\e-\langle v\e
\rangle_{{{{Y}}}_{i}\e}\right\|^2_{L_2({{{Y}}}_i\e)}\leq
C\eps^2\|\nabla v\e\|^2_{L_2({{{Y}}}_i\e)}\\\label{ineq11}
\eps^n\left|\langle v\e \rangle_{{{{Y}}}_{i}\e}-\langle
v\e\rangle_{ F_{i}\e}\right|^2\leq C\eps^{2}\|\nabla
v\e\|^2_{L_2({{{Y}}}_{i}\e)}\\\label{ineq2} \eps^n\left|\langle
v\e \rangle_{{{{Y}}}_{i}\e}-\langle v\e\rangle_{
R_{ij}\e}\right|^2\leq C\eps^{2}\|\nabla
v\e\|^2_{L_2({{{Y}}}_{i}\e)}\\\label{ineq3} \eps^n\left|\langle
v\e\rangle_{S_{ij}\e}-\langle v\e\rangle_{ R_{ij}\e}\right|^2\leq
C\eps^{2}\|\nabla v\e\|^2_{L_2( R_{ij}\e)}
\end{gather}
Inequality (\ref{ineq1}) is the Poincar\'{e} inequality,
inequalities (\ref{ineq11})-(\ref{ineq2}) follow directly from
Lemma \ref{lm20}. Let us prove inequality (\ref{ineq3}). We
introduce in $R_{ij}\e$ the spherical coordinates $(r,\Theta)$,
where $r$ is a distance to $x_{ij}\e$, $\Theta$ are the angle
coordinates. Below by $\mathrm{S}_{n-1}$ we denote the
$(n-1)$-dimensional unit sphere, by $d\Theta$ we denote the
Riemannian measure on $\mathrm{S}_{n-1}$. One has
$$v\e(r\e+\kappa\eps,\Theta)-v\e(r,\Theta)=\intl_{r}^{r\e+\kappa\eps} {\partial
v\e\over\partial\rho}(\rho,\Theta)d\rho,\quad r\in
(r\e,r\e+\kappa\eps) $$ We multiply this equality by $r^{n-1}dr
d\Theta$, integrate from $r\e$ to $r\e+\kappa\eps$ (with respect
to $r$) and over $\mathrm{S}_{n-1}$ (with respect to $\Theta$),
divide by $|R_{ij}\e|$ and square. Using the Cauchy inequality we
obtain
\begin{multline*}
\left|\langle v\e\rangle_{S_{ij}\e}-\langle v\e\rangle_{
R_{ij}\e}\right|^2=
\left|{1\over|R_{ij}\e|}\intl_{\mathrm{S}_{n-1}}\intl_{r\e}^{r\e+\kappa\eps}
\left(\intl_{r}^{r\e+\kappa\eps} {\partial
v\e\over\partial\rho}(\rho,\Theta)d\rho\right)r^{n-1}drd\Theta\right|^2\leq\\\leq
C \left(\intl_{\mathrm{S}_{n-1}}\intl_{r\e}^{r\e+\kappa\eps}\left|
{\partial v\e\over\partial\rho}(\rho,\Theta)\right|^2
\rho^{n-1}d\rho d\Theta\right)\cdot\left(
\intl_{r\e}^{r\e+\kappa\eps} {d\rho\over \rho^{n-1}}\right)\leq
C_1\|\nabla v\e\|^2_{L_2( R_{ij}\e)}\eps^{2-n}
\end{multline*}
and thus (\ref{ineq3}) is proved.

It is clear that (\ref{lm21_i1}) follows from (\ref{ineq1}),
(\ref{ineq2}), (\ref{ineq3}), and (\ref{lm21_i3}) follows from
(\ref{ineq1}), (\ref{ineq11}).
\end{proof}


\begin{lemma}\label{lm22} The following inequality is valid for an arbitrary $v\in
H^1(D_{ij}\e)$:
\begin{gather}\label{g_ineq}
\|v\|^2_{L_2(G_{ij}\e)}\leq C\eps^{\gamma-1}\left\{
\eta\e_{G_{ij}\e}[v]+\eps^2\eta_{R_{ij}\e}\e[v]+\|v\|^2_{L_2(R_{ij}\e)}\right\}
\end{gather}
\end{lemma}


\begin{proof}As in the proof of Lemma \ref{lm21} we introduce in
$G_{ij}\e$ the spherical coordinates $(r,\Theta)$. One has
\begin{gather}\label{star1}
v(r,\Theta)=v(r\e,\Theta)+\intl_{r\e}^{r}{\partial v\over
\partial \rho}(\rho,\Theta)d\rho,\ r\in (r\e-d\e,r\e)
\end{gather}
Taking into account (\ref{a}) we obtain from (\ref{star1})
\begin{multline*}
\intl_{\mathrm{S}_{n-1}}\intl_{r\e-d\e}^{r\e}|v(r,\Theta)|^2
r^{n-1}dr d\Theta\leq 2\left(\intl_{r\e-d\e}^{r\e}
{r^{n-1}dr}\right)\cdot\left((r\e)^{1-n}\intl_{\mathrm{S}_{n-1}}\left|v(r\e,\Theta)\right|^2
(r\e)^{n-1}d\Theta
 +\right.\\\left.
+\intl_{\mathrm{S}_{n-1}}\left(\intl_{r\e-d\e}^{r\e}\left|{\partial
v\over\partial\rho}(\rho,\Theta)\right|^2\rho^{n-1}d\rho\cdot
\intl_{r\e-d\e}^{r\e} {d\rho\over
\rho^{n-1}}\right)d\Theta\right)\leq
C\left(\eps^{\gamma}\|v\|^2_{L_2(\hat{C}_{ij}\e)}+\eps^{\gamma-1}\eta_{G_{ij}\e}^{\eps}[v]\right)
\end{multline*}
Similarly we obtain
\begin{gather*}
\|v\|^2_{L_2(\hat{C}_{ij})}\leq
C\left(\eps^{-1}\|v\|^2_{L_2(R_{ij}\e)}+\eps\|\nabla
v\|^2_{L_2(R_{ij}\e)}\right)
\end{gather*}
The statement of the lemma follows directly from the last two
inequalities.
\end{proof}


\begin{lemma}\label{lm23}
$\liml_{\eps\to 0}\lambda_1^{D,\eps}(D_{ij}\e)=\sigma_j$, where
$\sigma_j$ ($j=1,\dots,m$) are defined by (\ref{sigmarho}).
\end{lemma}


\begin{proof}Let
$v_{ij}\e\in\mathrm{dom}(\mathcal{A}^{D,\eps}_{D_{ij}\e})$ be the
eigenfunction corresponding to $\lambda_1^{D,\eps}(D_{ij}\e)$ and
such that
\begin{gather}
\label{v_mean} \ds\int_{B_{ij}\e}v_{ij}\e(x)dx=|B_{ij}\e|
\end{gather} Instead of calculating $v_{ij}\e$ in the exact form
we construct a convenient approximation $\mathbf{v}_{ij}\e$ for
it.

We introduce in $D_{ij}\e$ the spherical coordinates $(r,\Theta)$,
$r\in [0,r\e+\kappa\eps)$. Let $\varphi:\mathbb{R}\to \mathbb{R}$
be a twice-continuously differentiable function such that
$\varphi(\rho)=1$ as $\rho\leq 1/2$ and $\varphi(\rho)=0$ as
$\rho\geq 1$.

We define the function $\mathbf{v}_{ij}\e$ by the formula (below
we assume that ${3r\e\over 4}<r\e-d\e$ that is true for $\eps$
small enough)
\begin{gather}\label{hat_v}
\mathbf{v}_{ij}\e(r,\Theta)=\begin{cases}
1,& r\in\big[0,{r\e\over 2}\big)\\
 1+\check A_{j}^{\eps} r^{2-n}
\left(1-\varphi\left({|x-x_i\e|-r\e/2\over
r\e/4}\right)\right),&r\in
\big[{r\e\over 2}, r\e-d\e\big)\\
A_j\e r^{2-n}+B_j\e,& r\in\big[r\e-d\e,r\e\big)\\
\hat A_j\e r^{2-n}\varphi\left({|x-x_i\e|-r\e\over
\kappa\eps}\right),& r\in\big[r\e,r\e+\kappa\eps\big)
\end{cases}
\end{gather}
We choose the coefficients $A_j\e,\ \check A_j\e,\ \hat A_j\e,\
 B_j\e$ in such a way that $\mathbf{v}_{ij}\e$ satisfies conditions (\ref{sopr}):
\begin{gather*}
A_j\e={1\over
1-a_j\e}\left[(r\e-d\e)^{2-n}-(r\e)^{2-n}\right]^{-1}\sim{r^{n-1}\eps^{n-1-\gamma}\over
n-2}
\\
\check A_j\e=\hat A_j\e=a_j\e A_j\e,\quad B_j\e=A_j\e
(r\e)^{2-n}(a_j\e-1)
\end{gather*}
It is clear that $\mathbf{v}_{ij}\e\in
\mathrm{dom}(\mathcal{A}_{D_{ij}\e}^{D,\eps})$ and $\mathcal{A}\e
\mathbf{v}_{ij}\e=0\text{ in }D_{ij}\e\setminus\left\{x:\
|x-x_{ij}\e|\in \left[{5r\e\over 8},{3r\e\over
4}\right]\cup\left[r\e+{\kappa\eps\over
2},r\e+\kappa\eps\right]\right\}$.

Direct calculations lead to the following asymptotics as $\eps\to
0$:
\begin{gather}\label{vbf_est1}
\eta_{D_{ij}\e}\e[\mathbf{v}_{ij}\e]\sim  a_j |\partial
{B}_j|\eps^n,\quad
\|\mathbf{v}_{ij}\e\|_{L_{2,b\e}({B_{ij}\e})}^2\sim
b_j|{B}_j|\eps^n\\ \label{vbf_est2}\|\mathcal{A}\e
\mathbf{v}_{ij}\e\|_{L_2(D_{ij}\e)}=O(\eps^n),\quad
\|\mathbf{v}_{ij}\e-1\|_{L_2(B_{ij}\e)}^2+\|\mathbf{v}_{ij}\e\|_{L_2(G_{ij}\e\cup
R_{ij}\e)}^2=o(\eps^n)
\end{gather}

Using the min-max principle we get
\begin{gather}\label{courant}
\lambda^{D,\eps}_{1}(D_{ij}\e)=\ds{\eta\e_{D_{ij}\e}[v_{ij}\e]\over\|v_{ij}\e\|_{L_{2,b\e}
(D_{ij}\e)}^2}\leq
\ds{\eta_{D_{ij}\e}\e[\mathbf{v}_{ij}\e]\over\|\mathbf{v}_{ij}\e
\|_{L_{2,b\e}(D_{ij}\e)}}\sim {a_j |\partial {B}_j|\over
b_j|{B}_j|}={n a_j \over r b_j}=\sigma_j
\end{gather}

One has the following estimates for the eigenfunction $v_{ij}\e$:
\begin{gather}\label{v_est1}\|v_{ij}\e\|^2_{L_2(R_{ij}\e)}\leq C\eps^2
\eta_{R_{ij}\e}\e[v_{ij}\e]\\\label{v_est2}
\|v_{ij}\e-1\|^2_{L_2(B_{ij}\e)}\leq C\eps^2
\eta_{B_{ij}\e}\e[v_{ij}\e]\\\label{v_est3}
\|v_{ij}\e\|^2_{L_2(G_{ij}\e)}\leq C\eps^{\gamma-1}\left\{
\eta\e_{G_{ij}\e}[v_{ij}\e]+\eps^2\eta_{R_{ij}\e}\e[v_{ij}\e]+\|v_{ij}\e\|^2_{L_2(R_{ij})}\right\}
\end{gather}
The first one is the Friedrichs inequality, the second one is the
Poincar\'{e} inequality and the third one follows from Lemma
\ref{lm22}. Furthermore one has the equality
\begin{gather}\label{v_est4}
\eta\e_{D_{ij}\e}[v_{ij}\e]=\lambda^{D,\eps}_{1}(D_{ij}\e)\left(\|v_{ij}\e\|^2_{L_2(R_{ij}\e)}+b_j\|v_{ij}\e\|^2_{L_2(G_{ij}\e)}
+b_j\left(\|v_{ij}\e-1\|^2_{L_2(B_{ij}\e)}+|B_{ij}\e|\right)\right)
\end{gather}
It follows from (\ref{courant})-(\ref{v_est4}) that
\begin{gather}\label{v_est5}
\eta\e_{D_{ij}\e}[v_{ij}\e]=O(\eps^{n}),\quad
\|{v}_{ij}\e-1\|_{L_2(B_{ij}\e)}^2+\|{v}_{ij}\e\|_{L_2(G_{ij}\e\cup
R_{ij}\e)}^2=o(\eps^n)\text{ as }\eps\to 0
\end{gather}
Moreover (\ref{v_mean}), (\ref{v_est5}) imply
\begin{gather}\label{v_est6}
\|{v}_{ij}\e\|^2_{L_{2,b\e}({B_{ij}\e})}\sim b_j|{B}_j|\eps^n
\end{gather}

Now let us estimate the difference
$w_{ij}\e=v_{ij}\e-\mathbf{v}_{ij}\e$. One has
\begin{gather*}
\|w_{ij}\e\|^2_{L_2(D_{ij}\e)}\leq
2\left(\|v_{ij}\e\|^2_{L_2(G_{ij}\e\cup
R_{ij}\e)}+\|\mathbf{v}_{ij}\e\|^2_{L_2(G_{ij}\e\cup
R_{ij}\e)}\right)+2\left(\|v\e_{ij}-1\|^2_{L_2(B_{ij}\e)}+\|1-\mathbf{v}_{ij}\e\|^2_{L_2(B_{ij}\e)}\right)
\end{gather*}
and thus in view of (\ref{vbf_est2}), (\ref{v_est5}) we conclude
that
\begin{gather}\label{w_est1}
\|w_{ij}\e\|^2_{L_2(D_{ij}\e)}=o(\eps^n)
\end{gather}
Furthermore using inequality (\ref{courant}) we get
\begin{gather*}
\eta_{D_{ij}\e}\e[w_{ij}\e]\leq -2(\mathcal{A}\e
\mathbf{v}_{ij}\e,w\e)_{L_{2,b\e}({D_{ij}\e})}+\left({\eta_{D_{ij}\e}\e[\mathbf{v}_{ij}\e]\over
\|\mathbf{v}_{ij}\e\|^2_{L_{2,b\e}({D_{ij}\e})}}\|{v}_{ij}\e\|^2_{L_{2,b\e}({D_{ij}\e})}-
\eta_{D_{ij}\e}\e[\mathbf{v}_{ij}\e]\right)
\end{gather*}
and in view of (\ref{vbf_est1}), (\ref{vbf_est2}),
(\ref{v_est5})-(\ref{w_est1}) we conclude that
\begin{gather}\label{w_est2}
\eta_{D_{ij}\e}\e[w_{ij}\e]=o(\eps^n)
\end{gather}

The statement of the lemma follows directly from (\ref{vbf_est1}),
 (\ref{w_est1}), (\ref{w_est2}).
\end{proof}


\begin{lemma}\label{lm24} $\liml_{\eps\to 0}\lambda_2^{D,\eps}(D_{ij}\e)=\infty$
\end{lemma}


\begin{proof}
We denote:
\begin{gather*}
\mathbf{B}\e=\left\{y\in \mathbb{R}^n:\ 0\leq
|y|<r-\eps^{\gamma-1}\right\},\quad \mathbf{B}=\left\{y\in
\mathbb{R}^n:\ 0\leq |y|<r\right\}\\
\mathbf{G}\e=\left\{y\in
\mathbb{R}^n:\ r-\eps^{\gamma-1}<|y|<r\right\}\\
\mathbf{R}=\left\{y\in \mathbb{R}^n:\ r<|y|<r+\kappa\right\},\quad
\mathbf{D}=\left\{y\in \mathbb{R}^n:\ 0\leq|y|<r+\kappa\right\}
\end{gather*}
Also we introduce the functions $\mathbf{a}\e(y)$,
$\mathbf{b}(y)$:
$$\mathbf{a}\e(y)=a\e(y\eps+x_{ij}\e),\quad \mathbf{b}(y)=b\e(y\eps+x_{ij}\e),\quad y\in\mathbf{D}$$
(it is clear that $\mathbf{b}$ in independent of $\eps$).

By ${\mathbf{A}}^{D,\eps}_{\mathbf{D}}$ we denote the operator
acting in $L_{2,\mathbf{b}}(\mathbf{D})$ and being defined by the
operation
\begin{gather*}{\mathbf{A}}^{D,\eps}_{\mathbf{D}}=
-{1\over \mathbf{b}(y)}\suml_{k=1}^n\ds{\partial\over\partial
y_k}\left(\mathbf{a}\e(y){\partial \over\partial y_k}\right)
\end{gather*}
and the definitional domain
$\mathrm{dom}({\mathbf{A}}^{D,\eps}_{\mathbf{D}})$ which consists
of functions $v$ belonging to $H^2(\mathbf{B}\e)$,
$H^2(\mathbf{G}\e)$, $H^2(\mathbf{R})$ and satisfying the
conditions
\begin{gather*}
\begin{cases} (v)^+=(v)^-\text{\quad and\quad } \ds\left({\partial
v\over\partial n}\right)^+=a_j\e\left({\partial v\over\partial
n}\right)^-,
&y\in\partial\mathbf{B}\\
(v)^+=(v)^-\text{\quad and\quad } a_j\e\ds\left({\partial
v\over\partial n}\right)^+=\left({\partial v\over\partial
n}\right)^-, &y\in\partial \mathbf{B}\e\\
v=0,& y\in\partial \mathbf{D}
\end{cases}
\end{gather*}
We denote by ${\lambda}_k^{{D},\eps}(\mathbf{D})$ the $k$-th
eigenvalue of the operator ${\mathbf{A}}^{D,\eps}_{\mathbf{D}}$.
It is clear that
\begin{gather}\label{elambda}
\forall k\in \mathbb{N}:\quad
{\lambda}_k^{{D},\eps}(\mathbf{D})=\eps^{2}
{\lambda}_k^{{D},\eps}(D_{ij}\e)
\end{gather}

Below we will prove that
\begin{gather}\label{conv1}
\forall k\in \mathbb{N}:\quad{\lambda}_k^{{D},\eps}(\mathbf{D})\to
{\lambda}_k
\end{gather}
where ${\lambda}_k$ is the $k$-th eigenvalue of the operator
${\mathbf{A}}$ which acts in the space $L_2(\mathbf{R})\oplus
L_{2,b_j}(\mathbf{B})$ and is defined by the formula
$${\mathbf{A}}=
-\left(\begin{matrix}\Delta_{\mathbf{R}}^{D,N}
&0\\0&b_j^{-1}\Delta_{\mathbf{B}}^N
\end{matrix}\right)$$
Here the operator $\Delta_{\mathbf{R}}^{D,N}$ (\textit{resp.}
$\Delta_{\mathbf{B}}^N$) is defined by the operation $\Delta$ and
the definitional domain consisting of functions $v\in
H^2(\mathbf{R})$ (\textit{resp.} $v\in H^2(\mathbf{B})$)
satisfying the conditions
\begin{gather*}
v|_{\partial \mathbf{D}}=0,\quad \left.{\partial v\over\partial
n}\right|_{\partial \mathbf{R} \setminus\partial
\mathbf{D}}=0\text{\quad (\textit{resp.} }  \left.{\partial
v\over\partial n}\right|_{\partial \mathbf{B}}=0\text{)}
\end{gather*}

It is clear that ${\lambda}_1=0$ (${\lambda}_1$ coincides with the
first eigenvalue of $-b_j^{-1}\Delta_{\mathbf{B}}^N$) while
\begin{gather}\label{lambda2}
{\lambda}_2>0
\end{gather}
(${\lambda}_2$ coincides either with the first eigenvalue of
$-\Delta_{\mathbf{R}}^{D,N}$ or with the second eigenvalue of
$-b_j^{-1}\Delta_{\mathbf{B}}^N$). Then the statement of the lemma
follows directly from (\ref{elambda})-(\ref{lambda2}).

To complete the proof of lemma we have to prove (\ref{conv1}). For
that we use the following

\begin{theorem*}[see \cite{IOS}] Let $\mathcal{H}^\eps, \mathcal{H}^0$ be
separable Hilbert spaces, let $\mathcal{L}\e:\mathcal{H}\e\to
\mathcal{H}\e,\ \mathcal{L}^0:\mathcal{H}^0 \to \mathcal{H}^0$ be
linear continuous operators,
$\mathrm{im}\mathcal{L}^0\subset\mathcal{V}\subset \mathcal{H}^0$,
where $\mathcal{V}$ is a subspace in $\mathcal{H}^0$.

Suppose that the following conditions $C_1-C_4$ hold:

{$C_1.$} The linear bounded operators $R\e:\mathcal{H}^0\to
\mathcal{H}\e$ exist such that $ \|R\e
f\|^2_{\mathcal{H}\e}\underset{\eps\to 0}\to
\varrho\|f\|^2_{\mathcal{H}^0}$ for any $f\in \mathcal{V}$. Here
$\varrho>0$ is a constant.

{$C_2.$} Operators $\mathcal{L}\e, \mathcal{L}^0$ are positive,
compact and self-adjoint. The norms
$\|\mathcal{L}\e\|_{\mathcal{L}(\mathcal{H}\e)}$ are bounded
uniformly in $\eps$.

{$C_3.$} For any $f\in \mathcal{V}$: $\|\mathcal{L}\e R\e
f-R\e\mathcal{L}^0 f\|_{\mathcal{H}\e}\underset{\eps\to 0}\to 0$.

{$C_4.$} For any sequence $f\e\in \mathcal{H\e}$ such that
$\sup\limits_{\eps} \|f\e\|_{\mathcal{H}\e}<\infty$ the
subsequence $\eps^\prime\subset\eps$ and $w\in \mathcal{V}$ exist
such that $ \|\mathcal{L}\e f\e-R\e
w\|_{\mathcal{H}\e}\underset{\eps=\eps^\prime\to 0}\longrightarrow
0$.

Then for any $k\in\mathbb{N}$\ $$\mu_k\e\underset{\eps\to
0}\to\mu_k$$ where $\{\mu_k\e\}_{k=1}^\infty$ and
$\left\{\mu_k\right\}_{k=1}^\infty$ are the eigenvalues of the
operators $\mathcal{L}\e$ and $\mathcal{L}^0$, which are
renumbered in the increasing order and with account of their
multiplicity.
\end{theorem*}

Let us apply this theorem. We set
$\mathcal{H}\e=L_{2,\mathbf{b}}(\mathbf{D})$,
$\mathcal{H}^0=L_2(\mathbf{R})\oplus L_{2,b_j}(\mathbf{B})$,
$\mathcal{L}\e=({\mathbf{A}}_{\mathbf{D}}^{D,\eps}+\mathrm{I})^{-1}$,
$\mathcal{L}^0=(\mathbf{A}+\mathrm{I})^{-1}$,
$\mathcal{V}=\mathcal{H}^0$. We introduce the operator
$R\e:\mathcal{H}^0\to \mathcal{H}\e$ by the formula
\begin{gather*}
[R\e f](y)=\begin{cases}f_\mathbf{R}(y),&y\in \mathbf{R} ,\\
f_\mathbf{B}(y),&y\in \mathbf{B},
\end{cases}\quad
f=(f_\mathbf{R},f_\mathbf{B})\in \mathcal{H}^0
\end{gather*}

Evidently conditions $C_1$ (with $\varrho=1$) and $C_2$ hold. Let
us verify condition $C_3$.

At first we introduce the operator $Q\e:H^1(\mathbf{B}\e)\to
H^1(\mathbb{R}^n)$ by the formula
$$[Q\e v](y)=[Q\tilde{v}\e](k\e y)$$ where $k\e=(r-\eps^{\gamma-1})^{-1}r$,
the function $\tilde{v}\e\in H^1(\mathbf{B})$ is defined by the
formula $\tilde{v}\e(y)=v(y/k\e)$ and $Q:H^1(\mathbf{B})\to
H^1(\mathbb{R}^n)$ is the operator with the following properties:
$$\forall v\in H^1(\mathbf{B}):\quad [Q v](y)=v(y)\text{ for }y\in \mathbf{B},\quad
\|Q v\|_{H^1(\mathbb{R}^n)}\leq C\|v\|_{H^1(\mathbf{B})}$$ (such
an operator exists, see e.g. \cite{Mikhailov}). One has
\begin{gather*}
\forall v\in H^1(\mathbf{B}\e):\quad [Q\e v](y)=v(y)\text{ for
}y\in \mathbf{B}\e
\end{gather*}
Since $k\e\sim 1$ as $\eps\to 0$, then, obviously,
\begin{gather}\label{Q2}
\forall v\in H^1(\mathbf{B}\e):\quad \|Q\e
v\|_{H^1(\mathbb{R}^n)}\leq C_1\|v\|_{H^1(\mathbf{B}\e)}
\end{gather}

Let $f=(f_\mathbf{R},f_\mathbf{B})\in\mathcal{H}^0$. We set
$f\e=R\e f$, $v\e=\mathcal{L}\e f\e$. It is clear that
\begin{gather}\label{est1}
\|v\e\|_{L_{2,\mathbf{b}}(\mathbf{D})}\leq
\|f\e\|_{L_{2,\mathbf{b}}(\mathbf{D})}=\|f\|^2_{\mathcal{H}^0}
\end{gather}
One has the following integral equality:
\begin{gather}
\label{int_ineq_D}
\intl_{\mathbf{D}}\bigg[\mathbf{a}\e(y)\big(\nabla v\e,\nabla
w\e\big)+\mathbf{b}(y)\left(v\e w\e-f\e w\e\right)\bigg]dy=0,\quad
\forall w\e\in \overset{\circ}{H}{}^1(\mathbf{D})
\end{gather}
Substituting into (\ref{int_ineq_D}) $w\e=v\e$ and taking into
account (\ref{est1}) we obtain
\begin{gather}\label{est2}
\intl_{\mathbf{D}}\mathbf{a}\e|\nabla v\e|^2 dy\leq C
\end{gather}

Let $v\e_\mathbf{R}\in H^1(\mathbf{R})$ (\textit{resp.}
$v\e_{\mathbf{B}}\in H^1(\mathbf{B})$) be the restrictions of
$v\e$ onto $\mathbf{R}$ (\textit{resp.} the restrictions of $Q\e
v\e$ onto $\mathbf{B}$). Since $v\e\in \mathrm{dom}(
\mathbf{A}^{D,\eps}_{\mathbf{D}})$ then
$v\e_\mathbf{R}|_{\partial\mathbf{D}}=0$. It follows from
estimates (\ref{Q2}), (\ref{est1}), (\ref{est2}) that the set
$\left\{(v_{\mathbf{R}}\e, v_{\mathbf{B}}\e)\right\}_{\eps}$ is
bounded in $H^1(\mathbf{R})\oplus H^1(\mathbf{B})$ uniformly in
$\eps$. Therefore the set
$\left\{(v_\mathbf{R}\e,v_\mathbf{B}\e)\right\}_\eps$ is weakly
compact in $H^1(\mathbf{R})\oplus H^1(\mathbf{B})$ and in view of
the embedding theorem it is compact in $L_2(\mathbf{R})\oplus
L_2(\mathbf{B})$. Let $\eps^\prime\subset\eps$ be an arbitrary
subsequence for which
\begin{gather}\label{conv_RB}
\begin{array}{l}
v_{\mathbf{R}}\e\underset{\eps=\eps^\prime\to 0}\longrightarrow
v_\mathbf{R}\in H^1(\mathbf{R})\text{ weakly in
}H^1(\mathbf{R})\text{ and strongly in }L_2(\mathbf{R}),\quad
v_\mathbf{R}|_{\partial\mathbf{D}}=0\\
v_{\mathbf{B}}\e\underset{\eps=\eps^\prime\to 0}\longrightarrow
v_\mathbf{B}\in H^1(\mathbf{B})\text{ weakly in
}H^1(\mathbf{B})\text{ and strongly in }L_2(\mathbf{B})
\end{array}
\end{gather}
We will prove that
\begin{gather}\label{vA0}
v=\mathcal{L}^0 f,\text{ where }v=(v_\mathbf{R},v_\mathbf{B})
\end{gather}

We define the function $w\e\in \overset{\circ}{H}{}^1(\mathbf{D})$
by the formula
\begin{gather*}
w\e(x)=\left(w_\mathbf{B}(x)-w_\mathbf{R}(x)\right)\varphi\left({|x-x_{ij}\e|-
(r-\eps^{\gamma-1})\over \eps^{\gamma-1} }\right)+w_\mathbf{R}(x)
\end{gather*}
Here $w_\mathbf{R},w_\mathbf{B}\in C^\infty(\mathbb{R}^n)$ are
arbitrary functions, $\mathrm{supp} (w_\mathbf{R})\subset
\mathbf{D}$, $\varphi:\mathbb{R}\to \mathbb{R}$ be a smooth
function such that $\varphi(\rho)=1$ as $\rho\leq 1/2$ and
$\varphi(\rho)=0$ as $\rho\geq 1$. Substituting $w\e$ into
(\ref{int_ineq_D}) we get
\begin{gather}\label{int_ineq_D1}
\intl_{\mathbf{R}}\bigg[\big(\nabla v_\mathbf{R}\e,\nabla
w_\mathbf{R}\big)+v_\mathbf{R}\e w_\mathbf{R}-f_\mathbf{R}
w_\mathbf{R}\bigg]dy+\intl_{\mathbf{B}}\bigg[\big(\nabla
v_\mathbf{B}\e,\nabla w_\mathbf{B}\big)+ b_j\left(v_\mathbf{B}\e
w_\mathbf{B}-f_\mathbf{B}
w_\mathbf{B}\right)\bigg]dy+\delta(\eps)=0
\end{gather}
where
\begin{gather*}
\delta(\eps)=-\intl_{\mathbf{G}\e}\bigg[\big(\nabla
v_\mathbf{B}\e,\nabla w_\mathbf{B}\big)+b_j\left(v_\mathbf{B}\e
w_\mathbf{B}-f_\mathbf{B}
w_\mathbf{B}\right)\bigg]dy+\intl_{\mathbf{G}\e}\bigg[a_j\e\big(\nabla
v\e,\nabla w\e\big)+b_j\big(v\e w\e-f\e w\e\big)\bigg]dy
\end{gather*}
It is clear that
$$\intl_{\mathbf{G}\e}a_j\e|\nabla w\e|^2dy+\|w\e\|_{L_2(\mathbf{G}\e)}^2\leq
C(\eps^2+\eps^{\gamma-1})$$ and due to (\ref{Q2}), (\ref{est1}),
(\ref{est2}) we get $\delta\e\to 0$ as $\eps\to 0$. Taking into
account (\ref{conv_RB}) we pass to the limit as
$\eps=\eps^\prime\to 0$ in (\ref{int_ineq_D1}) and obtain
\begin{gather*}
\intl_{\mathbf{R}}\bigg[\big(\nabla v_\mathbf{R},\nabla
w_\mathbf{R}\big)+v_\mathbf{R} w_\mathbf{R}-f_\mathbf{R}
w_\mathbf{R}\bigg]dy+\intl_{\mathbf{B}}\bigg[\big(\nabla
v_\mathbf{B},\nabla w_\mathbf{B}\big)+ b_j\left(v_\mathbf{B}
w_\mathbf{B}-f_\mathbf{B} w_\mathbf{B}\right)\bigg]dy=0
\end{gather*}
Hence $-\Delta^{D,N}_{\mathbf{R}} v_R+v_R=f_R$ and
$-b^{-1}_j\Delta^{N}_{\mathbf{B}} v_B+ v_B=f_B$. Therefore
(\ref{vA0}) holds. In view of (\ref{vA0})
$(v_\mathbf{R},v_\mathbf{B})$ is independent of the subsequence
$\eps^\prime$ and thus $(v_\mathbf{R}\e,v_\mathbf{B}\e)$ converges
to $(v_\mathbf{R},v_\mathbf{B})$ as $\eps \to 0$.

Making the substitution $x=y\eps+x_{ij}\e$ in estimate
(\ref{g_ineq}) we get
\begin{gather*}
\|v\e\|^2_{L_2(\mathbf{G}\e)}\leq
C\eps^{\gamma-1}\left\{\eps^{-2}\intl_{\mathbf{G}\e}a_j\e |\nabla
v\e|^2dy +\intl_{\mathbf{R}}|\nabla v\e|^2dy
+\|v\e\|^2_{L_2(\mathbf{R})}\right\}
\end{gather*}
and therefore in view of (\ref{est1}), (\ref{est2}) we obtain
(recall that $\gamma>3$)
\begin{gather}\label{v_ets_G}
\|v\e\|^2_{L_2(\mathbf{G}\e)}\underset{\eps\to 0}\to 0
\end{gather}

Taking into account (\ref{conv_RB}), (\ref{vA0}), (\ref{v_ets_G})
we get
$$\|\mathcal{L}\e R\e f-R\e \mathcal{L}^0 f\|_{\mathcal{H}\e}^2\leq
\|v_\mathbf{R}\e-v_\mathbf{R}\|_{L_2(\mathbf{R})}^2+
\|v_\mathbf{B}\e-v_\mathbf{B}\|_{L_{2,b_j}(\mathbf{B}\e)}^2+
2\left(\|v\e\|_{L_{2}(\mathbf{G}\e)}^2+\|v_\mathbf{B}\|_{L_{2}(\mathbf{G}\e)}^2\right)
\underset{\eps\to 0}\to 0$$ and thus $C_3$ is proved.

Finally let us verify condition $C_4$. Let $\sup\limits_{\eps}
\|f\e\|_{\mathcal{H}\e}<\infty$. We denote $v\e=\mathcal{L}\e
f\e$, it is clear that the set $\{v\e\}_\eps$ is bounded in
$H^1(\mathbf{D})$ uniformly in $\eps$. Then the set
$\left\{(v\e_\mathbf{R},v\e_\mathbf{B})\right\}_\eps$ is bounded
in $H^1(\mathbf{R})\oplus H^1(\mathbf{B})$ uniformly in $\eps$ and
therefore the subsequence $\eps^\prime\subset\eps$ and
$w=(w_\mathbf{R},w_\mathbf{B})\in {H}^1(\mathbf{R})\oplus
{H}^1(\mathbf{B})\subset \mathcal{H}^0$ exist such that
\begin{gather*}
v_{\mathbf{R}}\e\underset{\eps=\eps^\prime\to 0}\longrightarrow
w_\mathbf{R}\text{ weakly in
}H^1(\mathbf{R})\text{ and strongly in }H^1(\mathbf{R})\\
v_{\mathbf{B}}\e\underset{\eps=\eps^\prime\to 0}\longrightarrow
w_\mathbf{B}\text{ weakly in }L_2(\mathbf{B})\text{ and strongly
in }L_2(\mathbf{B})
\end{gather*}
Moreover $v\e$ satisfies (\ref{v_ets_G}), therefore
$\liml_{\eps=\eps^\prime\to 0}\|\mathcal{L}^0 f\e-w\|^2= 0$. $C_4$
is proved.

We have verified the fulfilment of conditions $C_1-C_4$. Thus the
eigenvalues $\mu_k\e$ of the operator $\mathcal{L}\e$ converge to
the eigenvalues $\mu_k$ of the operator $\mathcal{L}^0$ as
$\eps\to 0$. But
$\lambda_k^{D,\eps}(\mathbf{D})=(\mu_k\e)^{-1}-1$, $
\lambda_k=(\mu_k)^{-1}-1$ that implies (\ref{conv1}). The lemma is
proved.
\end{proof}


\section{\label{sec4}Structure of $\sigma(\mathcal{A}^0)$}

In this section we prove equality (\ref{Aspectrum}).

At first we show that
\begin{gather}\label{arg}
\lambda\in
\sigma(\mathcal{A}^0)\setminus\cupl_{j=1}^m\left\{\sigma_j\right\}
\ \Longleftrightarrow\
\lambda\mathcal{F}(\lambda)\in\sigma(\widehat{\mathcal{A}})
\end{gather}
where $\sigma(\widehat{\mathcal{A}})$ is the spectrum of the
operator $\widehat{\mathcal{A}}=-\ds\suml_{k,l=1}^n
\widehat{a}^{kl}{\partial^2 \over\partial x_k\partial x_l }$, the
function $\mathcal{F}(\lambda)$ is defined by (\ref{mu_eq})

Indeed let
$\lambda\in\sigma(\mathcal{A}^0)\setminus\cupl_{j=1}^m\left\{\sigma_j\right\}
$. Then there is nonzero
$F=\left(\begin{matrix}f\\f_1\\\dots\\f_m\end{matrix}\right)\in
L_2(\mathbb{R}^n)\underset{j=\overline{1,m}}\oplus
L_{2,{\rho_j/\sigma_j}}(\mathbb{R}^n)$ such that
\begin{gather}\label{opposite} F\notin
\mathrm{im}(\mathcal{A}^0-\lambda \mathrm{I})
\end{gather}
Let us suppose the opposite, i.e. $ \lambda
\mathcal{F}(\lambda)\notin\sigma(\widehat{\mathcal{A}})$. Then for
any $g\in L_2(\mathbb{R}^n)$ there is $u\in
\mathrm{dom}(\widehat{\mathcal{A}})$ such that
\begin{gather}\label{pencil}
\widehat{\mathcal{A}} u-\lambda\mathcal{F}(\lambda)u=g\end{gather}
We set $\ds g=f+\suml_{j=1}^m{\rho_j f_j\over\sigma_i-\lambda}$.
It follows from (\ref{pencil}) that
$$\mathcal{A}^0U-\lambda U=F,\text{\quad where
}U=\left(\begin{matrix}u\\u_1\\\dots\\u_m\end{matrix}\right),\
u_j=\ds{\sigma_j u+f_j\over\sigma_j-\lambda}\quad
(j={1,\dots,m})$$ We obtain a contradiction with (\ref{opposite}),
hence $\lambda \mathcal{F}(\lambda)\in
\sigma(\widehat{\mathcal{A}})$. Converse assertion in (\ref{arg})
is proved similarly.

It is well-known that $\sigma(\widehat{\mathcal{A}})=[0,\infty)$,
therefore
\begin{gather}\label{0infty}
\lambda\in
\sigma(\mathcal{A}^0)\setminus\cupl_{j=1}^m\left\{\sigma_j\right\}
\text{ iff }\lambda \mathcal{F}(\lambda)\geq 0
\end{gather}

At first we study the function $\lambda\mathcal{F}(\lambda)$ on
$\mathbb{R}$. It is easy to get (see Fig. \ref{fig2}) that there
are the points $\mu_j$, $j=1,\dots,m$ such that
\begin{gather*}
\mathcal{F}(\mu_j)=0,\ j=1,\dots,m-1\\
\sigma_j<\mu_j<\sigma_{j+1},\ j=1,\dots,m-1,\quad
\sigma_m<\mu_m<\infty\\
\left\{\lambda\in\mathbb{R}\setminus\cupl_{j=1}^m\left\{\sigma_j\right\}:\
\lambda\mathcal{F}(\lambda)\geq
0\right\}=[0,\infty)\setminus\left(\cupl_{j=1}^m[\sigma_j,\mu_j)\right)
\end{gather*}

\begin{figure}[t]
\begin{center}
\scalebox{0.45}[0.45]{\includegraphics{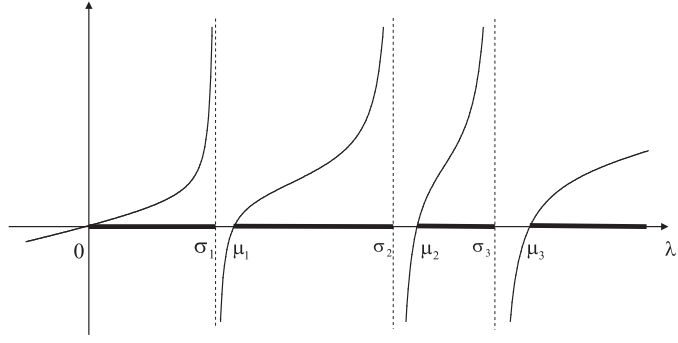}}
\caption{\label{fig2}The graph of the function
$\lambda\mathcal{F}(\lambda)$ (for $m=3$).}
\end{center}
\end{figure}

Let us consider the equation $\lambda\mathcal{F}(\lambda)=a$,
where $a\in [0,\infty)$. One the one hand it is equivalent to the
equation $\left(\prod\limits_{j=1}^m(\sigma_j-\lambda)\right)^{-1}
P_{m+1}(\lambda)=0$, where $P_{m+1}$ is a polynomial of the degree
$m+1$, and therefore in $\mathbb{C}$ this equation at most $m+1$
roots. On the other hand on $[0,\infty)$ the equation
$\lambda\mathcal{F}(\lambda)=a$ has exactly $m+1$ roots (see Fig.
\ref{fig2}). Thus the set $\{\lambda\in
\mathbb{C}:\lambda\mathcal{F}(\lambda)\geq 0\}$ belong to
$[0,\infty)$.

We conclude that $\lambda\in
\sigma(\mathcal{A}^0)\setminus\cupl_{j=1}^m\left\{\sigma_j\right\}$
iff $\lambda\in
[0,\infty)\setminus\left(\cupl_{j=1}^m[\sigma_j,\mu_j)\right)$.
Since $\sigma(\mathcal{A}^0)$ is a closed set then the points
$\sigma_j$, $j=\overline{1,m}$ also belong to
$\sigma(\mathcal{A}^0)$. This completes the proof of equality
(\ref{Aspectrum}).


\section{\label{sec5}Proof of Hausdorff convergence}

This section is a main part of the proof: we show that the set
$\sigma(\mathcal{A}\e)$ converges in the Hausdorff sense to the
set $\sigma(\mathcal{A}^0)$ as $\eps\to 0$, that is the following
conditions (\ref{ah}) and (\ref{bh}) hold:

\begin{gather}\tag{A$_\mathrm{H}$}\label{ah}
\text{if }\lambda\e\in\sigma(\mathcal{A}\e)\text{ and
}\liml_{\eps\to 0}\lambda\e=\lambda\text{ then }\lambda\in
\sigma(\mathcal{A}^0)\\\tag{B$_\mathrm{H}$}\label{bh} \text{for
any }\lambda\in \sigma(\mathcal{A}^0)\text{ there exists
}\lambda\e\in\sigma(\mathcal{A}\e)\text{ such that }\liml_{\eps\to
0}\lambda\e=\lambda
\end{gather}


\subsection{Proof of condition \eqref{ah}} Let
$\lambda\e\in\sigma(\mathcal{A}\e)$, $\liml_{\eps\to
0}\lambda\e=\lambda$. We have to prove that $\lambda\in
\sigma(\mathcal{A}^0)$.

If $\lambda\in\cupl_{j=1}^m\left\{\sigma_j\right\}$ then
(\ref{ah}) holds true since
$\cupl_{j=1}^m\left\{\sigma_j\right\}\subset\sigma(\mathcal{A}^0)$.
Therefore we focus on the case $\lambda\notin
\cupl_{j=1}^m\left\{\sigma_j\right\}$.

We consider the sequence $\eps_N\subset\eps$, where
$\eps_N={1\over N}$, $N=1,2,3$\dots\ For convenience we preserve
the same notation $\eps$ having in mind the sequence $\eps_N$.

Taking into account Remark \ref{rem21} we conclude that there
exists $\theta\e\in \mathbb{T}^n$ such that
$\lambda\e\in\sigma(\mathcal{A}^{\theta\e,\eps}_{{{{Y}}}})$. We
extract a subsequence (still denoted by $\eps$) such that
$\theta\e\underset{\eps\to 0}\to\theta\in \mathbb{T}^n$.

Let $u\e\in \mathrm{dom}(\mathcal{A}^{\theta\e,\eps}_{{{{Y}}}})$
be the eigenfunction corresponding to $\lambda\e$ and such that
\begin{gather}\label{u_estimates}
\|u\e\|_{L_{2,b\e}({{{Y}}})}=1\text{ (and therefore }
\eta\e_{{{{Y}}}}[u\e]=\lambda\e)
\end{gather}

We introduce the operator $\Pi\e: H^1(F_{Y}\e)\to H^1({{{Y}}})$
such that for each $u \in H^1(F_{Y}\e)$:
\begin{gather}\begin{matrix}
\Pi\e u(x)=u(x)\ \text{ for }\ x\in F_{{{Y}}}\e\\
\|\Pi\e u\|_{H^1({{{Y}}})}\leq C\|u\|_{H^1(F_{{Y}}\e)}\end{matrix}
\end{gather}
It is known (see e.g. \cite{CioSJP,March}) that such an operator
exists.

Also we introduce the operators $\Pi_j\e:
L_2(\cup_{i\in\I}B_{ij}\e)\to L_2({{{Y}}})$ ($j=1,\dots,m$) by the
formula
\begin{gather*}
i\in\I,\ x\in{{{Y}}}_i\e:\ \Pi_j\e u(x)=\langle u
\rangle_{B_{ij}\e}
\end{gather*}
(recall that
$\overline{{{{Y}}}}=\cupl_{i\in\I}\overline{{{{Y}}}_i\e}$). Using
the Cauchy inequality we obtain
\begin{gather}\label{pij}
\|\Pi_j\e u\|_{L_2({{{Y}}})}\leq
C\|u\|_{L_2(\cup_{i\in\I}B_{ij}\e)}
\end{gather}

It follows from (\ref{u_estimates})-(\ref{pij}) and the embedding
theorem that a subsequence (still denoted by $\eps$), $u\in
H^1({{{Y}}})$ and $u_j\in L_2({{{Y}}})$ ($j=1,\dots,m$) exist such
that
\begin{gather*}
\Pi\e u\e\underset{\eps\to 0}\to u\text{ weakly in
}H^1({{{Y}}})\text{ and strongly in }L_2({{{Y}}})\\
\Pi_j\e u\e\underset{\eps\to 0}\to u_j\text{ weakly in
}L_2({{{Y}}})\quad (j=1,\dots,m)
\end{gather*}
Moreover due to the trace theorem
\begin{gather}\label{pi_trace_conv}
\Pi\e u\e\underset{\eps\to 0}\to u\text{ strongly in
}L_2(\partial{{{Y}}})
\end{gather}
and therefore $u$ belong to $H^1_\theta({{{Y}}})$, i.e.
\begin{gather}\label{u_cond}
\forall k=\overline{1,n}:\quad u(x+e_k)=\theta_k u(x),\text{ for }
x=\underset{^{\overset{\qquad\quad\uparrow}{\qquad\quad k\text{-th
place}}\qquad }}{(x_1,x_2,\dots,0,\dots,x_n)}
\end{gather}

We denote by $\widehat{\mathcal{A}}_{Y}^\theta$ the operator which
is defined by the operation $\widehat{\mathcal{A}}_{Y}^\theta
u=-\ds\suml_{k,l=1}^n \widehat{a}^{kl}{\partial^2 u\over\partial
x_k\partial x_l}$ and the definitional domain
$\mathrm{dom}(\widehat{\mathcal{A}}_{Y}^\theta)$ consisting of
functions belonging to $H^2({Y})$ and satisfying $\theta$-periodic
boundary conditions, i.e.
\begin{gather*}
\forall k=\overline{1,n}:\quad
\begin{cases}
u(x+e_k)=\theta_k u(x),\\\ds \suml_{l=1}^n \widehat{a}^{k
l}{\partial u\over\partial x_l}(x+e_k)=\theta_k \suml_{l=1}^n
\widehat{a}^{kl}{\partial u\over\partial x_l}(x),
\end{cases}\text{ for }x=\underset{^{\overset{\qquad\quad\uparrow}{\qquad\quad k\text{-th
place}}\qquad }}{(x_1,x_2,\dots,0,\dots,x_n)}
\end{gather*}
It is clear that
$\sigma(\widehat{\mathcal{A}}_{Y}^\theta)\subset[0,\infty)$.

\begin{lemma}\label{th51}
One has
\begin{gather}\label{lambdaA1}
u\in \mathrm{dom}(\widehat{\mathcal{A}}_{Y}^\theta)\text{ and }
\widehat{\mathcal{A}}_{Y}^\theta u=\lambda \mathcal{F}(\lambda)u
\end{gather}
\end{lemma}

\begin{proof}
One has the following integral equality:
\begin{gather}\label{int_ineq}
\intl_{{{{Y}}}}\bigg(a\e(x)\big(\nabla u\e(x),\nabla
\mathbf{w}\e(x)\big)-\lambda\e b\e(x) u\e(x)
\mathbf{w}\e(x)\bigg)dx=0,\quad\forall \mathbf{w}\e\in
H^1_{\overline{\theta\e}}({{{Y}}})
\end{gather}
In order to prove (\ref{lambdaA1}) we plug into (\ref{int_ineq}) a
function $\mathbf{w}\e$ of special type and then pass to the limit
as $\eps\to 0$.

We introduce some additional notations. Let $v_k\in C^2(F)$
($k=1,\dots,n$) be a function that solves the problem (\ref{cell})
in $F$. We denote by $\widehat v_k$ the function that belongs to
$C^2(Y)$ and coincides with $v_k$ in $F$ (such a function exists,
see e.g. \cite{Mikhailov}). Then we extend $\widehat v_k$ by
periodicity to the whole $\mathbb{R}^n$ preserving the same
notation for the extended function. Using a standard regularity
theory one can easily prove that $\widehat v_k\in
C^2(\mathbb{R}^n)$. We set
$$v_k^{\eps}(x)=\eps \widehat v_k(x\eps^{-1})$$

Let $\mathbf{v}_{ij}\e\in C^{2,\eps}(D_{ij}\e)$ ($i\in
\mathbb{Z}^n,\ j=1,\dots,m$) be the function which is defined in
$D_{ij}\e$ by formula (\ref{hat_v}), $\mathrm{supp}
(\mathbf{v}_{ij}\e)\subset D_{ij}\e$. We redefine it by zero in
$\mathbb{R}^n\setminus D_{ij}\e$. Recall that $\mathbf{v}_{ij}\e$
was constructed in Lemma \ref{lm23} as an approximation for the
eigenfunction ${v}_{ij}\e$ of the operator
$\mathcal{A}_{D_{ij}\e}^{D,\eps}$ which corresponds to the first
eigenvalue $\lambda_1^{D,\eps}(D_{ij}\e)$ and satisfies
(\ref{v_mean}).

Let $\varphi:\mathbb{R}\to \mathbb{R}$ be a twice-continuously
differentiable function such that $\varphi(\rho)=1$ as $\rho\leq
1/2$ and $\varphi(\rho)=0$ as $\rho\geq 1$. We set
$$\phi_i\e(x)=1-\suml_{j=1}^m\phi\left({|x-x_{ij}\e|-(r\e-d\e)\over d\e}\right),\ x\in \mathbb{R}^n$$
Its clear that
\begin{gather}\label{phi_properties}\begin{matrix}
\phi\e_i(x)=0\text{ for }x\in \cupl_{j=1}^m B_{ij}\e,\quad
\phi\e_i(x)=1\text{ for }x\in
\mathbb{R}^n\setminus\left(\cupl_{j=1}^m(\overline{G_{ij}\e\cup
B_{ij}\e})\right)\\ |D^\a \phi_i\e|< C\eps^{-\a\gamma}\
(|\a|=0,1,2,\dots)\text{ in }\cupl_{j=1}^mG_{ij}\e\end{matrix}
\end{gather}

We cover $\mathbb{R}^n$ by the cubes
$$\widetilde{{{{Y}}}}_i^{\eps}=\left\{x\in \mathbb{R}^n: i\eps<x_k<(i+1)\eps+\eps^{3/2}\right\},\ i\in \mathbb{Z}^n$$
Let $\left\{\psi_i\e(x)\right\}_{i\in \mathbb{Z}^n}$ be a
partition of unity associated with this covering, that is
$$\psi_i\e(x)\in C^2(\mathbb{R}^n),\quad 0\leq \psi\e_i(x)\leq 1,\quad
\suml_{i\in \mathbb{Z}^n}\psi\e_i(x)=1,\quad \psi_i\e(x)=1\text{
if }x\in \widetilde{{{{Y}}}}_i^{\eps}\setminus\cupl_{l\not=
i}\widetilde{{{{Y}}}}_l^{\eps},\quad \psi_i\e(x)=0\text{ if
}x\notin\widetilde{{{{Y}}}}_i^{\eps}$$ Moreover, analyzing a
standard procedure of the construction of the partition of unity
(see e.g. \cite{Nara}) we can easily construct the partition of
unity satisfying the following additional conditions
\begin{gather}\label{psi_per}
\forall i\in \mathbb{Z}^n,\forall x\in \mathbb{R}^n:\
\psi_{i}\e(x)=\psi\e_{0}(x+i\eps)\\\label{psi_der} D^\a
\psi\e_0(x)<\ds{C\eps^{-{3\a/ 2}}}\ (\a=0,1,2)\text{ for }x\in
\widetilde{Y}_0\cap\left(\cupl_{l\not=0}\widetilde{Y}_l\right)
\end{gather}

We consider the function $w\e$ of the following form:
\begin{gather}\label{we}
w\e(x)=\suml_{i\in
\mathbb{Z}^n}\psi\e_i(x)\left(g\e_{i}(x)+h\e_{i}(x)\right)
\end{gather}
where
\begin{gather}\notag
\ds g_i\e(x)=g(x^{i,\eps})+\phi_i\e(x)\left(\suml_{k=1}^n
{\partial g\over
\partial x_k}(x^{i,\eps})\big(x_k-x_k^{i,\eps}-v_k\e(x)\big)+\right.
\\\label{we1}\ds\quad\quad\left.+{1\over 2}\suml_{k,l=1}^n {\partial^2 g\over
\partial x_k\partial
x_l}(x^{i,\eps})\big(x_k-x_k^{i,\eps}-v_k\e(x)\big)\big(x_l-x_{l}^{i,\eps}-v_l\e(x)\big)\right)
\\\label{we2}
h_i\e(x)=\suml_{j=1}^m
\big(h_j(x_{ij}\e)-g(x_{ij}\e)\big)\mathbf{v}_{ij}\e
\end{gather}
Here $x^{i,\eps}$ is the center of ${Y}_i\e$, $g(x)$, $h_j(x)$ are
arbitrary functions from $C^2(\mathbb{R}^n)$ satisfying
\begin{gather}\label{w_per}
\forall x\in \mathbb{R}^n,\ \forall i=(i_1,\dots,i_n)\in
\mathbb{Z}^n:\
\begin{cases} g(x+i)=\overline{\theta_1^{i_1}}\cdot\ldots\cdot
\overline{\theta_n^{i_n}} g(x)\\
h_j(x+i)=\overline{\theta_1^{i_1}}\cdot\ldots\cdot
\overline{\theta_n^{i_n}} h_j(x),\ j=1,\dots,m
\end{cases}
\end{gather}
Remark that $\ds{\partial g_i\e\over\partial n}=0$ on $\partial
G_{ij}\e$. Taking this into account we conclude that ${w}\e(x)$
belongs to $C^{2,\eps}(\mathbb{R}^n)$
 and in view of (\ref{psi_per}), (\ref{w_per}) and the
periodicity of $v^{\eps}_k$ we get
$$\forall x\in \mathbb{R}^n,\ \forall i\in \mathbb{Z}^n:\
w\e(x+i)=\overline{\theta_1^{i_1}}\cdot\ldots\cdot
\overline{\theta_n^{i_n} }w\e(x)$$ We also introduce the notations
$$g\e=\suml_{i\in
\mathbb{Z}^n}\psi\e_i(x)g\e_{i}(x),\quad h\e=\suml_{i\in
\mathbb{Z}^n}\psi\e_i(x)h\e_{i}(x)$$

The function $w\e$ belong to $H^1_{\overline{\theta}}({{{Y}}})$.
In order to obtain the function from
$H^1_{\overline{\theta\e}}({{{Y}}})$ we modify $w\e$ multiplying
it by the function which is very close to $1$ in ${Y}$ as $\eps\to
0$. Namely, we define the function $\mathbf{1}\e\in
C^\infty(\mathbb{R}^n)$ by the following recurrent formulae:
\begin{gather*}
\mathbf{1}\e(x_1,\dots,x_{n})=A_n(x_1,\dots,x_{n-1}) x_{n}+B_n(x_1,\dots,x_{n-1}),\\
\a=2,\dots,n:\
\begin{cases}B_\a(x_1,\dots,x_{\a-1})=A_{\a-1}(x_1,\dots,x_{\a-2})
x_{\a-1}+B_{\a-1}(x_1,\dots,x_{\a-2}),\\
A_\a(x_1,\dots,x_{\a-1})=\big({\overline{\theta_\a\e}}/{{\overline{\theta_\a}}}-1\big)B_\a(x_1,\dots,x_{\a-1}),\end{cases}\\
B_1=1,\ A_1={\overline{\theta_1\e}}/{\overline{\theta_1}}-1.
\end{gather*}
It is easy to see that
\begin{gather}\label{1e_1}\begin{matrix}
\maxl_{x\in\overline{{{{Y}}}}}\left|\mathbf{1}\e(x)-1\right|+\maxl_{x\in\overline{{{{Y}}}}}\left|\nabla
\mathbf{1}\e(x)\right|\underset{\eps\to 0}\to 0\\
\mathbf{1}^\eps\in
H^1_{\overline{\theta\e}/\overline{\theta}}({{{Y}}})\text{, where
}\overline{\theta\e}/\overline{\theta}:=(\overline{\theta_1\e}/\overline{\theta_1},
\dots,\overline{\theta_n\e}/\overline{\theta_n}) \end{matrix}
\end{gather}

Finally we set
\begin{gather*}
\mathbf{w}\e(x)=w\e(x)+(\mathbf{1}\e(x)-1) w\e(x)
\end{gather*}
It is clear that $\mathbf{w}\e\in
H^1_{\overline{\theta\e}}({{{Y}}})$.

Substituting $\mathbf{w}\e$ into (\ref{int_ineq}) and integrating
by parts we obtain
\begin{multline}\label{int_ineq+}
\intl_{{{{Y}}}}\bigg(u\e\mathcal{A}\e w\e-\lambda\e  u\e
w\e\bigg)b\e dx+\\+\intl_{\partial{{{Y}}}}u\e{\partial
w\e\over\partial n} ds+\intl_{{{{Y}}}}{\bigg(a\e\big(\nabla
u\e,\nabla ((\mathbf{1}\e-1)w\e)\big)-\lambda\e b\e
u\e(\mathbf{1}\e-1)w\e \bigg)dx}=0
\end{multline}

Further we will prove that the second and the third integrals in
(\ref{int_ineq+}) tend to zero as $\eps\to 0$. Now we focus on the
first integral in (\ref{int_ineq+}). Using Lemma \ref{lm22} and
taking into account (\ref{u_estimates}), (\ref{phi_properties}) we
obtain the estimates
\begin{gather}\label{star2}
\|u\e\|^2_{L_2(G_{Y}\e)}\leq
C\eps^{\gamma-1}\suml_{i\in\I}\left(\eta_{G_{ij}\e}\e[u\e]+\eps^2\eta_{R_{ij}\e}\e[u\e]+\|u\e\|^2_{L_2(R_{ij}\e)}\right)\leq
C_1\eps^{\gamma-1}
\\\label{star3}
\|\mathcal{A}\e g\e\|^2_{L_2(G_{Y}\e)}\leq C\eps^{3-\gamma}
\end{gather}
Since $\mathcal{\mathcal{A}}\e h\e=0$ in $G_{ij}\e$ then in view
of (\ref{equiv}), (\ref{star2}), (\ref{star3})
\begin{gather}\label{vanish_G}
\left|(u\e,\mathcal{\mathcal{A}}\e
w\e)_{L_{2,b\e}(G_{Y}\e)}\right|\leq C\eps\underset{\eps\to 0}\to
0
\end{gather}
Similarly we obtain
\begin{gather}\label{vanish_G1}
\liml_{\eps\to 0}(u\e,w\e)_{L_{2,b\e}(G\e_{{{Y}}})}=0
\end{gather}

We denote $$\widetilde F_i\e=\left\{x\in F_i\e: i
\eps+\eps^{3/2}<x_k<(i+1)\eps\right\},\  \widetilde
F_{Y}\e=\cupl_{i\in\I}\widetilde F_i\e$$ It is clear that
$\widetilde F_{i}\e=F_i\e\setminus\left(\cupl_{l\not= i}\widetilde
{Y}_l\e\right)$.

Firstly we estimate $g\e$ in $F_{Y}\e\setminus \widetilde
F_{Y}\e$. We represent $g\e$ in $F_{Y}\e\setminus \widetilde
F_{Y}\e$ in the form
\begin{multline}\label{repres_T}
g\e(x)=\suml_{i\in
\mathbb{Z}^n}\psi_i(x)\left[\left(g^{i,\eps}+\suml_{k=1}^n
g_{,k}^{i,\eps}\big(x_k-x^{i,\eps}_k\big)+{1\over
2}\suml_{k,l=1}^n
g^{i,\eps}_{,kl}\big(x_k-x^{i,\eps}_k\big)\big(x_l-x^{i,\eps}_l\big)-g(x)\right)-\right.\\\left.
-\suml_{k=1}^n \left( g^{i,\eps}_{,k}+\suml_{l=1}^n
g^{i,\eps}_{,kl}\big(x_l-x^{i,\eps}_l\big)-g_{,k}(x)\right)v_k\e(x)+{1\over
2}\suml_{k,l=1}
(g^{i,\eps}_{,kl}-g_{,kl}(x))v_k\e(x)v_l\e(x)\right]+\\+g(x)-\suml_{k=1}^n
g_{,k}(x)v_k\e(x)+ {1\over 2}\suml_{k,l=1}^n
g_{,kl}(x)v_k\e(x)v_l\e(x)\end{multline} Here
$g^{i,\eps}=g(x^{i,\eps})$, $\ds g_{,k}(x)={\partial
g\over\partial x_k}(x)$, $g_{,k}^{i,\eps}=g\e_{,k}(x^{i,\eps})$,
$\ds g_{,kl}(x)={\partial^2 g\over\partial x_k\partial x_l}(x)$,
$g_{,kl}^{i,\eps}=g\e_{,kl}(x^{i,\eps})$. It follows from
(\ref{psi_der}), (\ref{repres_T}) that $|\Delta g\e(x)|< C$ for
$x\in F_{Y}\e\setminus \widetilde F_{Y}\e$. Since
$\dist\left(\cupl_{j=1}^m D_{ij}\e,\partial
Y_i\e\right)\geq\kappa\eps$ then $h\e=0$ in $F_{Y}\e\setminus
\widetilde F_{Y}\e$ when $\eps$ is small enough and therefore
\begin{gather}\label{vanish_T}
\left|(u\e,\mathcal{\mathcal{A}}\e
w\e)_{L_{2,b\e}(F_{Y}\e\setminus \widetilde F_{Y}\e)}\right|\leq
C\|\Delta g\e\|_{L_2(F_{Y}\e\setminus \widetilde F_{Y}\e)}\leq
C_1\sqrt{|F_{Y}\e\setminus \widetilde F_{Y}\e|}\leq C_2\eps^{1/4}
\end{gather}
Similarly we obtain
\begin{gather}\label{vanish_T1}
\liml_{\eps\to 0}(u\e,w\e)_{L_{2,b\e}(F_{Y}\e\setminus \widetilde
F_{Y}\e)}=0
\end{gather}

Let us study $g\e$ and $h\e$ in $\widetilde F\e$. It is clear that
\begin{gather}\label{delta}
\Delta g\e=\suml_{k,l=1}^n g^{i,\eps}_{,kl}
\big(\nabla(x_k-v\e_k), \nabla(x_l-v\e_l)\big)\text{ for }x\in
\widetilde{{F}}_i\e
\end{gather}
In view of Lemma \ref{lm20} and the Poincar\'{e} inequality one
has the following estimate:
\begin{gather}\label{delta1}
\left\|u\e-\langle u\e\rangle_{\widetilde
F_i\e}\right\|^2_{L_2(\widetilde F_i\e)}+\eps^n\left|\langle \Pi\e
u\e\rangle_{{Y}_i\e}-\langle u\e\rangle_{\widetilde
F_i\e}\right|^2+\left\|\Pi\e u\e-\langle \Pi\e u\e
\rangle_{{Y}_i\e}\right\|_{L_2({{Y}_i\e})}^2\leq
C\eps^2\left\|\nabla\Pi\e u\e\right\|_{L_2({Y}_i\e)}^2
\end{gather}
Using (\ref{delta}), (\ref{delta1}) and the Poincar\'{e}
inequality we get
\begin{multline}\label{vanish_F}
(u\e,\mathcal{A\e}w\e)_{L_{2,b\e}(\widetilde
F_{Y}\e)}=-\suml_{k,l=1}^n
\left[\left(\intl_{F}\big(\nabla(x_k-v_k),
\nabla(x_l-v_l)\big)dx\right)\eps^n \suml_{i\in\I}g^{i,\eps}_{,kl}
\langle \Pi\e u\e\rangle_{{Y}_i\e}\right]+o(1)=
\\=-\suml_{k,l=1}^n \widehat{a}^{kl}|F|\intl_{{{{Y}}}}u\e
{\partial^2 g\over\partial x_k\partial
x_l}dx+o(1) \underset{\eps\to 0}\to -\suml_{k,l=1}^n
\widehat{a}^{kl}|F|\intl_{{{{Y}}}} u {\partial^2 g\over\partial
x_k\partial x_l}dx
\end{multline}
In the same way using Lemma \ref{lm21} (for $v\e=\Pi\e u\e$) we
obtain
\begin{gather}\label{vanish_F1}
(u\e,g\e)_{L_{2,b\e}(\widetilde
F\e_{{{Y}}})}=\suml_{i\in\I}g(x^{i,\eps})\langle u\e
\rangle_{F_i\e}|F|\eps^n+o(1)\underset{\eps\to 0}\to
{|F|}\intl_{{Y}}u(x)g(x)dx
\end{gather}
(here we also use the estimate $\eps^n\|\langle u\e
\rangle_{\widetilde{F}_i\e}-\langle u\e \rangle_{F_i\e}\|^2\leq
C\eps^2\|\nabla\Pi\e u\e\|^2_{L_2(Y_i\e)}$ which follows from
Lemma \ref{lm20}).

Let us study $h\e$ in $\widetilde F_{{{Y}}}\e$. Integrating by
parts and taking into account the form of the function
$\mathbf{v}_{ij}\e$ (in particular, we have the estimate
$\|\mathcal{A}\e \mathbf{v}_{ij}\e\|_{L_2(Y_i\e)}<C\eps^n$), the
Poincar\'{e} inequality and Lemma \ref{lm21} we obtain
\begin{multline}\label{vanish_F2}
(u\e,\mathcal{A\e}h\e)_{L_{2,b\e}(\widetilde
F\e_{{{Y}}})}=(u\e,\mathcal{A\e}h\e)_{L_{2,b\e}(F\e_{{{Y}}})}=
\suml_{j=1}^m\suml_{i\in\I}\langle u\e\rangle_{F_{i}\e}\intl_{\hat
C_{ij}\e}{\partial \mathbf{v}_{ij}\e\over\partial
r}\left(h_j(x_{ij}\e)-g(x_{ij}\e)\right)
ds+\\+\suml_{i\in\I}\left(u\e-\langle
u\e\rangle_{F_{i}\e},\mathcal{A}\e
h_i\e\right)_{L_{2,b\e}(F_i\e)}=\suml_{j=1}^m a_j\left|\partial
{B}_j\right|\suml_{i\in\I}\langle
u\e\rangle_{F_{i}\e}\left(g(x_{ij}\e)-h_j(x_{ij}\e)\right)
\eps^n+o(1)\underset{\eps\to 0}\to\\\underset{\eps\to 0}\to
\suml_{j=1}^m a_j\left|\partial {B}_j\right|\intl_{{{{Y}}}}
u(x)\left(g(x)-h_j(x)\right)dx
\end{multline}
(here $r=|x-x_{ij}\e|$). In the same way we get
\begin{gather}\label{vanish_F3}
\liml_{\eps\to 0}(u\e,h\e)_{L_{2,b\e}(\widetilde F\e_{{{Y}}})}=0
\end{gather}

Let us study $h\e$ in $B_{{{Y}}}\e$ ($g\e=0$ in $B_{{{Y}}}\e$).
Integrating by parts and using the Poincar\'{e} inequality we
obtain
\begin{multline}\label{vanish_B}
(u\e,\mathcal{A\e}h\e)_{L_{2,b\e}(
B\e_{{{Y}}})}=-\suml_{j=1}^m\suml_{i\in\I}\langle u\e\rangle_{
B_{ij}\e}\intl_{\check C_{ij}\e}{\partial
\mathbf{v}_{ij}\e\over\partial
r}\left(h_j(x_{ij}\e)-g(x_{ij}\e)\right) ds+o(1)=\\=\suml_{j=1}^m
a_j\left|\partial {B}_j\right|\intl_{Y} \Pi_j\e u\e(x)
(h_j(x)-g(x))dx+o(1)\underset{\eps\to 0}\to \suml_{j=1}^m
a_j\left|\partial {B}_j\right|\intl_{{{{Y}}}}
u_j(x)\left(h_j(x)-g(x)\right)dx
\end{multline}
In the same way we get
\begin{gather}\label{vanish_B1}
\liml_{\eps\to 0}(u\e,h\e)_{L_{2,b\e}(B\e_{{{Y}}})}=\suml_{j=1}^m
|{B}_j| b_j\intl_{\Omega}u_j(x)h_j(x)dx
\end{gather}

Finally, let us estimate the remaining integrals in
(\ref{int_ineq+}). One can easily obtain that
$$\eta_{Y}\e[w\e]+\|w\e\|^2_{L_{2,b\e}({Y})}<C$$ and therefore in view of
(\ref{1e_1})
\begin{gather}\label{vanish_2}
\liml_{\eps\to 0}\intl_{Y}\bigg(a\e \big(\nabla u\e,\nabla
((\mathbf{1}\e-1)w\e)\big)-\lambda\e b\e u\e(\mathbf{1}\e-1)w\e
\bigg)dx=0
\end{gather}

It is easy to see that the function $p\e=\left.\ds{\partial
w\e\over\partial n}\right|_{\partial {Y}}$ is bounded in
$L_2(\partial {Y})$ uniformly in $\eps$ and therefore there is a
subsequence (still denoted by $\eps$) and $p\in L_2(\partial {Y})$
such that
\begin{gather}\label{p_conv}
p\e\underset{\eps\to 0}\to p\text{ weakly in }L_2(\partial {Y})
\end{gather}
Moreover it is clear that $\forall k=\overline{1,n}$: $p\e(x+
e_k)=-\overline{\theta_k\e} p\e(x)$ for
$x=\underset{^{\overset{\qquad\quad\uparrow}{\qquad\quad
k\text{-th place}}\qquad }}{(x_1,x_2,\dots,0,\dots,x_n)}$.
Therefore
\begin{gather}\label{p_cond}
\forall k=\overline{1,n}:\ \quad p(x+ e_k)=-\overline{\theta_k}
p(x)\text{ for
}x=\underset{^{\overset{\qquad\quad\uparrow}{\qquad\quad
k\text{-th place}}\qquad }}{(x_1,x_2,\dots,0,\dots,x_n)}
\end{gather}
Taking into account (\ref{pi_trace_conv}), (\ref{u_cond}),
(\ref{p_conv}), (\ref{p_cond}) we get
\begin{gather}\label{vanish_3}
\liml_{\eps\to 0}\intl_{\partial{{{Y}}}}u\e{\partial
w\e\over\partial n} ds=\intl_{\partial {Y}}u p\ ds=0
\end{gather}

Then taking into account (\ref{vanish_G}), (\ref{vanish_G1}),
(\ref{vanish_T}), (\ref{vanish_T1}),
(\ref{vanish_F})-(\ref{vanish_2}), (\ref{vanish_3}) we pass to the
limit in (\ref{int_ineq+}) and obtain the equality
\begin{multline}\label{int_result}
\intl_{{Y}}\left( -u(x)|F|\suml_{k,l=1}^n
\widehat{a}^{kl}{\partial^2 g\over\partial x_k\partial
x_l}(x)-\lambda|F|u(x)g(x)+\suml_{j=1}^m\bigg( a_j|\partial
B_j|\big(g(x)-h_j(x)\big)u(x)\right.+\\\left. +a_j|\partial
B_j|\big(h_j(x)-g(x)\big)u_j(x)-\lambda |{B}_j|b_j
u_j(x)h_j(x)\bigg)\right) dx=0
\end{multline}
Recall that $g,h_j\in C^2(\mathbb{R}^n)$ are arbitrary functions
satisfying (\ref{w_per}).

Plugging $g=0$, $h_j=0$ for $j\not= k$ into (\ref{int_result}) and
taking into account the equality $|\partial B_j|={ |B_j|n r^{-1}}$
we get
\begin{gather}\label{u_k}
u_k={\sigma_k \over \sigma_k-\lambda}u,\quad k\in \{1,\dots,m\}
\end{gather}
Then setting $h_j=0$ for all $j=1,\dots,m$, integrating by parts
and taking into account (\ref{u_k}) we get
\begin{gather}\label{u}
\intl_{{Y}}\left(\suml_{k,l=1}^n \widehat{a}^{kl}{\partial
u\over\partial x_k}{\partial g\over\partial x_l}-\lambda
\mathcal{F}(\lambda) ug\right)dx=0
\end{gather}
where the function $\mathcal{F}(\lambda)$ is defined by
(\ref{mu_eq}).

Equality (\ref{u}) is valid for an arbitrary $g$ belonging to
$C^\infty(\mathbb{R}^n)$ and satisfying (\ref{w_per}). It is clear
that the set of such functions is dense in
$H_{\overline{\theta}}^1({Y})$. Therefore equality (\ref{u})
implies (\ref{lambdaA1}). Lemma \ref{th51} is proved.
\end{proof}

\begin{lemma}\label{th52}$u\not =0$.
\end{lemma}

\begin{proof} Let us introduce the spherical coordinates $(r,\Theta)$ in $D_{ij}\e$
and the function $u_{ij}\e$ by the formula
\begin{gather*}
u_{ij}\e(\rho,\Theta)=\langle u\e\rangle_{S\e_{ij}(\rho)},\text{
where }S\e_{ij}(\rho)=\left\{x\in \mathbb{R}^n:\
|x-x_{ij}\e|=\rho\right\}
\end{gather*}
One has the following Poincar\'{e} inequality:
\begin{gather*}
\|u\e-u_{ij}\e\|^2_{L_2(S\e_{ij}(\rho))}\leq
C\rho^2\|\nabla_{\Theta} u\e\|^2_{L_2(S\e_{ij}(\rho))}\leq
C_1\eps^2\|\nabla u\e\|^2_{L_2(S\e_{ij}(\rho))}
\end{gather*}
(here $\nabla_\Theta$ is a gradient on $S_{ij}\e(\rho)$: for
example in the case $n=2$ one has $\nabla_\Theta u
={1\over\rho^2}{\partial u\over\partial\theta}$).
 Integrating it by $\rho$ from $0$ to $r\e-d\e$ and summing by
$i$ we get
\begin{gather}\label{poincare}
\suml_{i\in\I}\|u\e-u_{ij}\e\|^2_{L_2(B\e_{ij})} \leq C\eps^2
\|\nabla u\e\|^2_{L_2(\cupl_i B_{ij}\e)}\leq C_1\eps^2\lambda\e
\end{gather}

We denote ${\mathbf{u}}_{ij}\e=u_{ij}\e-\langle
u\e\rangle_{S_{ij}\e}$. Clearly
$\mathbf{u}_{ij}\e\in\mathrm{dom}(\mathcal{A}_{D_{ij}\e}^{D,\eps})$
and
\begin{gather*}
\mathcal{A}_{D_{ij}\e}^{D,\eps} \mathbf{u}_{ij}\e-\lambda\e
\mathbf{u}_{ij}\e=\lambda\e\langle u\e\rangle_{S_{ij}\e}
\end{gather*}
Recall that $\lambda\notin\cupl_{j=1}^m\{\sigma_j\}$. Then in view
of Lemmas \ref{lm23}, \ref{lm24}\quad
$\lambda\e\notin\sigma(\mathcal{A}_{D_{ij}\e}^{D,\eps})$ when
$\eps$ is small enough. Therefore we have the following expansion:
\begin{gather}\label{decomp1}
\mathbf{u}_{ij}\e=\suml_{k=1}^\infty I^k_{ij}(\eps),\text{ where }
I^k_{ij}(\eps)={v^D_k(D_{ij}\e)}{\left(\lambda\e\langle
u\e\rangle_{S_{ij}\e},
v^D_k(D_{ij}\e)\right)_{L_{2,b\e}(D_{ij}\e)}\over
\left\|v^D_k(D_{ij}\e)\right\|_{L_{2,b\e}(D_{ij}\e)}^2
\left(\lambda^{D,\eps}_k(D_{ij}\e)-\lambda\e\right)}
\end{gather}
Here $\left\{v^D_k(D_{ij}\e)\right\}_{k=1}^m$ is a system of
eigenfunctions of $\mathcal{A}_{D_{ij}\e}^{D,\eps}$ corresponding
to $\left\{\lambda^{D,\eps}_k(D_{ij}\e)\right\}_{k=1}^m$ and such
that
$\left(v^D_k(D_{ij}\e),v^D_l(D_{ij}\e)\right)_{L_{2,b\e}(D_{ij}\e)}=0$
if $k\not= l$.

Using Lemmas \ref{lm21}, \ref{lm24} we get (for
$j\in\left\{1,\dots,m\right\}$)
\begin{gather}\label{sum2345}
\suml_{i\in\I}\left\|\suml_{k=2}^\infty
I_{ij}^k(\eps)\right\|^2_{L_2(B_{ij}\e)}\leq C
\maxl_{k=\overline{2,\infty}}
\left|\lambda^{D,\eps}_k(D_{ij}\e)-\lambda\e\right|^{-2}
\suml_{i\in\I}\left|\langle
u\e\rangle_{S_{ij}\e}\right|^2\eps^n\underset{\eps\to 0}\to 0
\end{gather}

As in Lemma \ref{lm23} we denote $v_{ij}\e=v^D_1(D_{ij}\e)$
assuming that $v_{ij}\e$ is normalized by condition
(\ref{v_mean}). Using estimates (\ref{v_est3}), (\ref{v_est5}) and
Lemma \ref{lm21} we get
\begin{gather}\label{sum1}
\suml_{i\in\I}\left\|I_{ij}^1(\eps)\right\|^2_{L_2(B_{ij}\e)}\sim
\suml_{i\in \I}{\lambda^2|{B}_j|\left|\langle
u\e\rangle_{S_{ij}\e}\right|^2\over(\sigma_j-\lambda)^2}\eps^n\sim{\lambda^2|{B}_j|
\|u\|^2_{L_2({Y})}\over(\sigma_j-\lambda)^2}
\end{gather}
as $\eps\to 0$. It follows from (\ref{decomp1})-(\ref{sum1}) that
\begin{gather}\label{v_lim}
\liml_{\eps\to
0}\suml_{i\in\I}\left\|\mathbf{u}_{ij}\e\right\|^2_{L_2(B_{ij}\e)}=
{\lambda^2|B_j|\|u\|^2_{L_2({Y})}\over(\sigma_j-\lambda)^2}
\end{gather}
Similarly we obtain
\begin{gather}\label{v_lim_star}
\intl_{B_{ij}\e} \mathbf{u}_{ij}\e dx\sim {\langle
u\e\rangle_{S_{ij}\e}\lambda
|B_j|\over\sigma_j-\lambda}\eps^n\text{ as }\eps\to 0
\end{gather}

Using (\ref{v_est3}), (\ref{v_est5}), (\ref{v_lim}),
(\ref{v_lim_star}) and Lemma \ref{lm21} we get
\begin{gather}\notag
\suml_{i\in\I}\left\|u_{ij}\e\right\|^2_{L_{2,b\e}(B_{ij}\e)}=\suml_{i\in\I}\left(\left\|
\mathbf{u}_{ij}\e\right\|^2_{L_2(B_{ij}\e)}+2\langle
u\e\rangle_{S_{ij}\e}\intl_{B_{ij}\e}\mathbf{u}_{ij}\e(x)
dx+\left|\langle u\e\rangle_{S_{ij}\e}\right|^2\cdot
|B_{j}\e|\eps^n\right)b_j\underset{\eps\to
0}\to\\\label{last}\underset{\eps\to
0}\to\left[{\lambda^2|{B}_j|\over(\sigma_j-\lambda)^2}+
{2\lambda|{B}_j|\over\sigma_j-\lambda}+|{B}_j|\right]b_j\|u\|^2_{L_2({Y})}=
|{B}_j|b_j\left(\sigma_j\over\sigma_j-\lambda\right)^2\|u\|_{L_2({Y})}^2
\end{gather}

Using the Poincar\'{e} inequality and Lemma \ref{lm21} one can
easily prove that
\begin{gather}\label{uF1}
\|u\e\|^2_{L_{2,b\e}(F_{Y}\e)}=|F|\suml_{i\in\I}\langle
u\e\rangle_{F_i\e}\eps^n+o(1)\underset{\eps\to 0}\to  |F|\cdot
\|u\|^2_{L_2({Y})}
\end{gather}
Furthermore in view of Lemma \ref{lm22}
\begin{gather}\label{uF2}
\liml_{\eps\to 0}\|u\e\|^2_{L_{2,b\e}(G_{Y}\e)}=0
\end{gather}

Finally taking into account (\ref{poincare}),
(\ref{last})-(\ref{uF2}) we obtain
\begin{gather*}
1=\|u\e\|_{L_{2,b\e}({Y})}^2\underset{\eps\to 0}\to
\|u\|^2_{L_2({Y})}\left[|{F}|+\suml_{j=1}^m
\left(\sigma_j\over\sigma_j-\lambda\right)^2 |B_j|b_j\right]
\end{gather*}
and therefore $u\not= 0$. Lemma \ref{th52} is proved.
\end{proof}

It follows from Lemmas \ref{th51}-\ref{th52} that $\lambda
\mathcal{F}(\lambda)$ belong to the spectrum
$\sigma(\widehat{\mathcal{A}}^\theta_Y)$ of the operator
$\widehat{\mathcal{A}}^\theta_Y$. Therefore $\lambda
\mathcal{F}(\lambda)\in [0,\infty)$ and in view of (\ref{0infty})
$\lambda\in\sigma(\mathcal{A}^0)$. Condition (\ref{ah}) is proved.


\subsection{Proof of condition \eqref{bh}} Let
$\lambda\in\sigma(\mathcal{A}^0)$. Let us prove that there is
$\lambda\e\in \sigma(\mathcal{A}\e)$ such  that $\liml_{\eps\to
0}\lambda\e=\lambda$.

We assume the opposite: the subsequence (still denoted by $\eps$)
and $\delta>0$ exist such that
\begin{gather}
\label{dist} \dist(\lambda,\sigma(\mathcal{A}\e))>\delta.
\end{gather}

Since $\lambda\in\sigma(\mathcal{A}^0)$  then the function
$F=\left(\begin{matrix}f\\f_1\\\dots\\f_m\end{matrix}\right)\in
L_2(\mathbb{R}^n)\underset{j=\overline{1,m}}\oplus
L_{2,{\rho_j/\sigma_j}}(\mathbb{R}^n)$ exists such that
\begin{gather}\label{notinim}
F\notin \mathrm{im}(\mathcal{A}-\lambda\mathrm{I}),\text{ where
}\mathrm{I}\text{ is the identical operator}\end{gather}

It follows from (\ref{dist}) that
$\lambda\in\mathbb{R}\setminus\sigma(\mathcal{A}\e)$. Then
$\mathrm{im}(\mathcal{A}\e-\lambda\mathrm{I})=L_{2,b\e}(\mathbb{R}^n)$
and hence for an arbitrary $f\e\in L_{2,b\e}(\mathbb{R}^n)$ there
is the unique $u\e\in \mathrm{dom}(\mathcal{A}\e)$ such that
\begin{gather}
\label{bvp1} \mathcal{A}\e u\e-\lambda u\e=f\e
\end{gather}

We substitute the following $f\e(x)\in L_{2,b\e}(\mathbb{R}^n)$
into (\ref{bvp1}):
\begin{gather*}
f\e(x)=\begin{cases}\langle f\rangle_{Y_i\e},& x\in F_i\e,\\\ds
\langle f_j\rangle_{Y_i\e},& x\in B_{ij}\e,\\0,& x\in
\cupl_{i.j}G_{ij}\e.
\end{cases}
\end{gather*}
It is clear that the norms $\|f\e\|_{L_{2,b\e}(\mathbb{R}^n)}$ are
bounded uniformly in $\eps$. Then in view of (\ref{dist}) $u\e$
satisfies the inequality
\begin{gather*}
\|u\e\|_{L_{2,b\e}(\mathbb{R}^n)}\leq
\delta^{-1}{\|f\e\|_{L_{2,b\e}(\mathbb{R}^n)}}\leq C
\end{gather*}
Furthermore
\begin{gather*}
\|\nabla u\e\|^2_{L_{2}(\mathbb{R}^n)}\leq
\|f\e\|_{L_{2,b\e}(\mathbb{R}^n)}\cdot
\|u\e\|_{L_{2,b\e}(\mathbb{R}^n)}+|\lambda|\cdot
\|u\e\|_{L_{2,b\e}(\mathbb{R}^n)}^2\leq C
\end{gather*}
Hence a subsequence (still denoted by $\eps$) and $u\in
H^1(\mathbb{R}^n)$, $u_j\in L_2(\mathbb{R}^n)$ such that
\begin{gather*}
\Pi\e u\e\rightarrow u\text{ weakly in }
H^1(\mathbb{R}^n)\text{ and strongly in }L_2(G)\text{ for any compact set }G\subset\mathbb{R}^n\\
\Pi_j\e u\e\rightarrow u_j\text{ weakly in } L_2(\mathbb{R}^n)\
(j=1,\dots,m)
\end{gather*}
where $\Pi\e$, $\Pi_j\e$ ($j=1,\dots,m$) are the operators
introduced above in the proof of condition (\ref{ah}).

For an arbitrary ${w}\e\in
\overset{\circ}{C}{}^\infty(\mathbb{R}^n)$ one has the following
integral equality:
\begin{gather}\label{int_ineq_f}
\intl_{{{\mathbb{R}^n}}}\bigg(a\e(x)\big(\nabla u\e(x),\nabla
{w}\e(x)\big)-\lambda\e b\e(x) u\e(x) {w}\e(x)-b\e(x)
f\e(x){w}\e(x)\bigg)dx=0
\end{gather}
We substitute into (\ref{int_ineq_f}) the function ${w}\e$ of the
form (\ref{we})-(\ref{we2}), but with $g,h_j\in
\overset{\circ}{C}{}^\infty(\mathbb{R}^n)$. Making the same
calculations as in the proof of condition (\ref{ah}) we obtain
\begin{multline}\label{int_ineq_final}
\intl_{\mathbb{R}^n}\bigg[ -u(x)\suml_{k,l=1}^n
\widehat{a}_{kl}|F|{\partial^2 g\over\partial
x_k\partial x_l}(x)-\lambda|F|u(x)g(x)-|F|f(x)g(x)+\\
+\suml_{j=1}^m\bigg(
a_j|\partial{B}_j|\big(g(x)-h_j(x)\big)u(x)+\\
+a_j|\partial{B}_j|\big(h_j(x)-g(x)\big)u_j(x)-\lambda |{B}_j|b_j
u_j(x)h_j(x)-|{B}_j|b_jf_j (x)h_j(x)\bigg)\bigg] dx=0
\end{multline}
for an arbitrary $g,h_j\in
\overset{\circ}{C}{}^\infty(\mathbb{R}^n)$. It follows from
(\ref{int_ineq_final}) that
\begin{gather*}
U=\left(\begin{matrix}u\\u_1\\\dots\\u_m\end{matrix}\right)\in\mathrm{dom}
(\mathcal{A}^0)\text{\quad and\quad }\mathcal{A}^0U-\lambda U=F
\end{gather*}
We obtain a contradiction with (\ref{notinim}).  Condition
(\ref{bh}) is proved.


\section{\label{sec6}End of proof of Theorem \ref{th1}}

In general the Hausdorff convergence of $\sigma(\mathcal{A}\e)$ to
$\sigma(\mathcal{A}^0)$ does not imply
(\ref{th1_f1})-(\ref{th1_f2})\footnote{For example, the set
$\sigma\e:=\sigma(\mathcal{A}^0)\cap\bigg(\cupl_{k\in
\mathbb{N}}\left[\eps k,\eps(k+{1\over 2})\right]\bigg)$ also
converges to $\sigma(\mathcal{A}^0)$ in the Hausdorff sense, but
the number of gaps in $\sigma\e\cap[0,L]$ tends to infinity as
$\eps\to 0$.}. However if we prove that $\sigma(\mathcal{A}\e)$
has at most $m$ gaps in $[0,L]$ when $\eps$ is less some $\eps_L$
then this implication holds true. More precisely the following
simple proposition is valid.

\begin{proposition}\label{prop1}
Let
$\mathcal{B}\e=[0,L]\setminus\left(\cupl_{j=1}^{m\e}(\a_j\e,\b_j\e)\right)$,
$\mathcal{B}=[0,L]\setminus\left(\cupl_{j=1}^{m}(\a_j,\b_j)\right)$,
where $L<\infty$ and
\begin{gather*}
0\leq\a_1\e,\quad \a\e_{j}<\b\e_{j}\leq \a\e_{j+1},\
j=\overline{1,m\e-1},\quad \a\e_{m\e}\leq L\\ 0<\a_1,\quad
\a_{j}<\b_{j}<\a_{j+1},\
j=\overline{1,m-1},\quad \a_{m}<L\\
m\e\leq m\\
\mathcal{B}\e\text{ converges to }\mathcal{B}\text{ in the
Hausdorff sense as }\eps\to 0
\end{gather*}

Then $m\e=m$ when $\eps$ is small enough and
$$\forall j=1,\dots,m:\quad \liml_{\eps\to 0}\a_j\e=\a_j,\quad \liml_{\eps\to 0}\b_j\e=\b_j$$
\end{proposition}

We introduce the notation
$[a_{k}^{-}(\eps),a_{k}^{+}(\eps)]:=\cupl_{\theta\in
\mathbb{T}^n}\left\{\lambda_k^{\theta,\eps}({Y}_0\e)\right\} $.

\begin{lemma}\label{lm4}$\liml_{\eps\to 0}a_{m+1}^{+}(\eps)=\infty$
\end{lemma}

\begin{proof}  In the same way as in the proof Lemma \ref{lm24}
we obtain the following equality
\begin{gather}\label{relationship1}
\liml_{\eps\to 0}\eps^2\lambda_k^{N,\eps}({Y}_0\e)=\lambda_k,\
k=1,2,3...
\end{gather}
where $\left\{\lambda_k\right\}_{k\in\mathbb{N}}$ are the
eigenvalues of the operator $\mathbf{A}$ which acts in the space
$L_2(F)\underset{j=\overline{1,m}}\oplus L_{2,b_j}({B}_j)$ and is
defined by the operation
\begin{gather*}
\mathbf{A}=-\left(\begin{matrix}\Delta^N_{F}&0&\dots&0\\
0&b_1^{-1}\Delta^N_{{B}_1}&\dots&0\\\vdots&\vdots&\ddots&\vdots\\0&0&\dots&b_m^{-1}\Delta_{{B}_m}^N\end{matrix}\right)
\end{gather*}
(here $\Delta^N_{F}$ and $\Delta^N_{B_j}$ are the Neumann
Laplacians in $F$ and $B_j$). It is clear that $\lambda_j=0$ for
$j=1,\dots,m+1$ while $\lambda_{m+2}>0$. Then using
(\ref{relationship1}) and taking into account (\ref{enclosure}) we
get
$$\liml_{\eps\to 0}a_{m+2}^{-}(\eps)\geq\liml_{\eps\to 0}\lambda_{m+2}^N({Y}_0\e)=
\lambda_{m+2}\liml_{\eps\to 0}\eps^{-2}=\infty$$

Suppose that there is a subsequence (still denoted by $\eps$) such
that the numbers $a_{m+1}^{+}(\eps)$ are bounded uniformly in
$\eps$. Let the numbers $L,L_1$ be such that ${\mu_m}<L<L_1$ and
$a_{m+1}^{+}(\eps)<L$. Since $\liml_{\eps\to
0}a_{m+2}^{-}(\eps)=\infty$ then $a_{m+2}^{-}(\eps)>L_1$ when
$\eps$ is small enough. Hence $\sigma(\mathcal{A}\e)\cap
[L,L_1]=\varnothing$ when $\eps$ is small enough. We obtain a
contradiction with condition (\ref{bh}) of the Hausdorff
convergence. Thus $\liml_{\eps\to 0}a_{m+1}^{+}(\eps)= \infty$.
\end{proof}

Lemma \ref{lm4} implies that for an arbitrary $L>0$ the spectrum
$\sigma(\mathcal{A}\e)$ has at most $m$ gaps in the interval
$[0,L]$ when $\eps$ is small enough:
\begin{gather*}
\sigma(\mathcal{A}\e)\cap[0,L]=[0,L]\setminus
\cupl_{j=1}^{m\e}(\sigma_j\e,\mu_j\e)
\end{gather*}
where $(\sigma_j\e,\mu_j\e)\subset[0,L]$ are some pairwise
disjoint intervals, $m\e\leq m$. Here the intervals are renumbered
in the increasing order.

We have proved that $\sigma(\mathcal{A}\e)$ converges to
$[0,\infty) \setminus\left(\cupl_{j=1}^m(\sigma_j,\mu_j)\right)$
in the Hausdorff sense as $\eps\to 0$. Let $L$ be an arbitrary
number such that $L>\mu_m$.  Then, evidently,
$\sigma(\mathcal{A}\e)\cap [0,L]$ converges to $[0,L]
\setminus\left(\cupl_{j=1}^m(\sigma_j,\mu_j)\right)$ in the
Hausdorff sense. By Proposition \ref{prop1} $m\e=m$ when $\eps$ is
small enough and
\begin{gather*}
\forall j=1,\dots,m:\quad \lim_{\eps\to
0}\sigma_j\e=\sigma_j,\quad \lim_{\eps\to 0}\mu_j\e=\mu_j
\end{gather*}

Theorem \ref{th1} is proved.


\section{\label{sec7}Proof of Theorem \ref{th2}}

Substituting $a_j$, $b_j$ (\ref{exact_formula}) into
(\ref{sigmarho}) we get $$\sigma_j=\alpha_j$$  (i.e. the first
equality in (\ref{ab}) holds) and
\begin{gather}\label{prod}
\rho_j=({\b_j-\a_j})\prod\limits_{i=\overline{1,m}|i\not=
j}\ds\left({\b_i-\a_j\over  \a_i-\a_j}\right)
\end{gather}

Recall that $\mu_j$ ($j=\overline{1,m}$) are the roots of equation
(\ref{mu_eq}), therefore in order to prove the equalities
$\mu_j=\b_j$ ($j=\overline{1,m}$) we have to show that
\begin{gather}\label{system}
\forall k=1,\dots ,m:\ \suml_{j=1}^m{\rho_j\over \b_k-\a_j}=1
\end{gather}

Let us consider (\ref{system}) as a system of $m$ linear algebraic
equations ($\rho_j$, $j=1,\dots,m$ are unknowns). It is clear that
(\ref{system}) follows from the following
\begin{lemma}\label{lm1}
The system (\ref{system}) has the unique solution
$\rho_1,\dots,\rho_m$ which is defined by (\ref{prod}).
\end{lemma}
\begin{proof}  We prove the lemma by induction. For $m=1$ its
validity is obvious. Suppose that we have proved it for $m=N-1$.
Let us prove it for $m=N$.

Multiplying the $k$-th equation in (\ref{system}) ($k=1,\dots,N$)
by $\b_k-\a_N$ and then subtracting the $N$-th equation from the
first $N-1$ equations we obtain a new system
\begin{gather*}
\forall k=1,\dots, N-1:\ \suml_{j=1}^{N-1}{\hat\rho_j\over
\b_k-\a_j}=1
\end{gather*}
where the new variables $\hat\rho_j$, $j=1,\dots, N-1$ are
expressed in terms of $\rho_j$ by the formula
\begin{gather}\label{new}
\hat\rho_j:=\rho_j\ds{\a_N-\a_j\over \b_N-\a_j},\ j=1,\dots, N-1
\end{gather}
Hence $\hat\rho_j$, $j=\overline{1,N-1}$ satisfy the system
(\ref{system}) with $m=N-1$. By the induction
\begin{gather}\label{syst_sol_ind}
\hat\rho_j=(\b_j-\a_j)\prod\limits_{i=\overline{1,N-1}|i\not=
j}\ds\left({\b_i-\a_j\over  \a_i-\a_j}\right)
\end{gather}
It follows from (\ref{new}), (\ref{syst_sol_ind}) that $\rho_j$
($j=1,\dots,N-1$) satisfy (\ref{prod}) (with $m=N$). The validity
of this formula for $\rho_N$ follows easily from the symmetry of
system (\ref{system}). Lemma \ref{lm1} is proved.
\end{proof}
Theorem \ref{th2} is proved.

\section*{Acknowledgements} The author is deeply grateful to Professor E. Khruslov
for the helpful discussion. The work is partially supported by the
M.V. Ostrogradsky research grant for young scientists.


\end{document}